\documentclass[12pt,thmsa, reqno]{amsart}

\usepackage{setspace}

\usepackage{amssymb, euscript,  mathrsfs}
\usepackage[dvips]{graphics}
\usepackage[title]{appendix}

\input{amssym.def}

\input{amssym.tex}

\usepackage{lineno}

\usepackage{tikz}

\usetikzlibrary{matrix,arrows,decorations.pathmorphing}

\usepackage[all,cmtip]{xy}

\def\Cal{\mathcal}

\def\C{{\Cal C}}

\def\M{{\Cal M}}

\def\S{{\Cal S}}

\def\Z{\mathcal{Z}}

\def\vnk{V_{n,k}}

\def\bbr{{\Bbb R}}

\def\bbc{{\Bbb C}}

\def\bbs{{\Bbb S}}

\def\rank{{\hbox{\rm rank}}}

\def\tr{{\hbox{\rm tr}}}

\def\grad{{\hbox{\rm grad}}}

\def\span{{\hbox{\rm span}}}

\def\det{{\hbox{\rm det}}}

\def\min{{\hbox{\rm min}}}

\def\Pr{{\hbox{\rm Pr}}}

\def\gnk{G_{n,k}}

\def\vnk{V_{n,k}}

\def\rn{\bbr^n}
\def\sn{S^{n-1}}

\def\part{\partial}
\def\intl{\int\limits}
\def\b{\beta}

\def\Gam{\Gamma}
\def\Om{\Omega}
\def\a{\alpha}
\def\om{\omega}

\def\Del{\Delta}
\def\del{\delta}
\def\vp{\varphi}

\def\gam{\gamma}

\def\sig{\sigma}
\def\lam{\lambda}
\def\z{\zeta}
\def\th{\theta}

\def\t{\tau}

\def\chi{{\bf 1}}

\def\snm1{\bbs^{n-1}}

\font\frak=eufm10

\def\fr#1{\hbox{\frak #1}}

\def\frM{\fr{M}}

\def\Re{\mathrm{Re}\,}

\def\intl{\int\limits}

\def\tr{\mathrm{tr}}

\def\grad{\mathrm{grad}}

\def\gnk{\mathrm{G}_{n,k}}
\def\gnm{\mathrm{G}_{n,m}}
\def\vnk{\mathrm{V}_{n,k}}
\def\vnm{\mathrm{V}_{n,m}}

\def\sym{\mathrm{Sym}}

\newcommand{\Cos}{\mathrm{Cos}}

\def\Cs{\mathscr{C}}

\def\Ma{\frM_{n,m}}
\def\gm{\Gamma_m}

\def\Cs{\mathscr{C c 1234}}
\def\d{\partial}

\def\cdt{\stackrel{*}{T}\!{}_{m, k}^\lam}

\def\cd{\stackrel{*}{\C}\!{}_{m, k}^\lam}
\def\sd{\stackrel{*}{\S}\!{}_{m, k}^\lam}
\def\cd0{\stackrel{*}{\C}\!{}_{m, k}^\lam}
\def\sd0{\stackrel{*}{\S}\!{}_{m, k}^\lam}
\def\fd{\stackrel{*}{F}\!{}_{\!m, k}}

\def\ncd0{\stackrel{*}{\Cs}\!{}_{m, k}^\lam}

\newtheorem{theorem}{Theorem}[section]
\newtheorem{lemma}[theorem]{Lemma}

\theoremstyle{definition}

\theoremstyle{remark}
\newtheorem{remark}[theorem]{Remark}

\theoremstyle{corollary}
\newtheorem{corollary}[theorem]{Corollary}

\newtheorem{proposition}[theorem]{Proposition}
\newtheorem{conjecture}[theorem]{Conjecture}

\numberwithin{equation}{section}

\newcommand{\be}{\begin{equation}}
\newcommand{\ee}{\end{equation}}

\newcommand{\bea}{\begin{eqnarray}}
\newcommand{\eea}{\end{eqnarray}}
\newcommand{\Bea}{\begin{eqnarray*}}
\newcommand{\Eea}{\end{eqnarray*}}

\def\sideremark#1{\ifvmode\leavevmode\fi\vadjust{\vbox to0pt{\vss
 \hbox to 0pt{\hskip\hsize\hskip1em
\vbox{\hsize2cm\tiny\raggedright\pretolerance10000
 \noindent #1\hfill}\hss}\vbox to8pt{\vfil}\vss}}}%

\begin{document}


\title[The $\lam$-Cosine Transforms ]{The $\lam$-Cosine Transforms, Differential Operators, and Funk Transforms on Stiefel and Grassmann Manifolds}

\author{ B. Rubin}

\address{Department of Mathematics, Louisiana State University, Baton Rouge,
Louisiana 70803, USA}
\email{borisr@lsu.edu}




\maketitle

\begin{abstract} We introduce a new family of invariant differential operators associated with $\lam$-cosine and Funk-Radon transforms on Stiefel and Grassmann manifolds. These  operators reduce the order of the
$\lam$-cosine  transforms  and yield new inversion formulas. Intermediate Funk-cosine transforms corresponding to integration over  matrices of lower rank are studied. The main tools are  polar decomposition and  Fourier analysis on matrix space.

 \end{abstract}

\setcounter{tocdepth}{1}

\tableofcontents


\section{Introduction}


\setcounter{equation}{0}

\noindent In the present paper we introduce new invariant differential operators  that can be used in the study of $\lam$-cosine,  $\lam$-sine,  and Funk  transforms on   Stiefel and Grassmann manifolds. These operators  are generalizations of polynomials of the  Beltrami-Laplace operator in the inversion formulas for the classical Funk transform on the unit sphere.
 We recall that the $\lam$-cosine transform of a function $f$ on the unit sphere $\sn$ in $\rn$ is an integral operator of the form
\be\label{af000}    (\C^\lam
f)(u)=   \intl_{\sn}  f(v)
|u \cdot v|^{\lam} \,d_*v, \qquad u\in \sn, \ee
where integration is performed with respect to the rotation invariant probability Haar measure.
 The name {\it cosine transform} for $\lam =1$  is due to Lutwak \cite [p. 385]{Lu} and reflects the  fact that $|u \cdot v|$ is  the cosine of the smallest angle between the straight lines along $u$ and $v$.
   The associated operator
\be\label{af000r}    (F f)(u)=   \intl_{\{v\in \sn:\, u \cdot v =0\}}  f(v)
 \,d_u v \ee
is called the Funk transform or a spherical Radon transform of $f$.
 An extensive bibliography related to operators  (\ref {af000}),  (\ref {af000r}), their   generalizations, and applications can be found in  \cite {Ga06, OPR, Ru15}.

In recent decades there is an increasing interest to  analogues of  (\ref{af000}) and  (\ref{af000r}), when the lines along $u$ and $v$ are replaced by  higher dimensional linear subspaces. These generalizations lead to  integral operators  that take functions on the Grassmannian $G_{n,m}$ of $m$-dimensional linear subspaces of $\rn$ to functions on the similar Grassmannian  $G_{n,k}$. In place of the Grassmannians, one can take their orthonormal bases (or frames), which are elements of the Stiefel manifolds $V_{n,m}$ and $V_{n,k}$, respectively.

 Historically the first publications related to higher rank generalizations of   (\ref{af000}) and  (\ref{af000r}), were probably short articles by  Petrov \cite{P1}   and Matheron \cite{Mat1}, though close  mathematical objects on matrix spaces were studied before by  G{\aa}rding \cite{Ga}, Gindikin \cite{Gi},
  and some other authors, in particular,  in multivariate statistics.  The paper \cite{P1} deals with
   inversion of  Radon transforms on
 matrices and  Grassmannians, while \cite{Mat1} contains a famous  injectivity conjecture for the higher rank cosine transform with  $\lam =1$. This conjecture was disproved by  Goodey and Howard \cite{GH1, GH}.
  More information  and further references can be found in \cite {A, AGS, GGR,   GR, OP, OPR,  OR05, OR06, Ru13, SZ, Zh1, Zh09}.

 A great deal has been written about Radon transforms on affine Grassmann manifolds. This circle of problems lies beyond the scope of the present paper. Information  can be found in
    \cite{Go, GK2, H65, Ru04d, RW2, Str}.

 \vskip 0.2 truecm

 Let us describe the  contents of the paper and main results.

 \vskip 0.2 truecm

{\bf 1.} The following question  was asked by Alesker \cite{A0}:

{\it Given a complex number $\lam$, what  differential operator $D$ satisfies}
\be \label {bmuqa}
D \,\C^{\lam +2} =\C^{\lam}  \; ?\ee
 The answer  is known for   the unit sphere \cite [p. 285]{Ru15}, where $D$ is a polynomial of the Beltrami-Laplace operator and the reasoning relies on the  spherical harmonic technique.
  In the Grassmannian set-up, when  the lines along $u$ and $v$ in (\ref{af000}) are replaced by linear subspaces, say $\xi \in G_{n,m}$  and $\t \in  G_{n,k}$,
  the question was studied by   Alesker,  Gourevitch, and Sahi \cite{AGS}  for $k=m$.  The authors used the tools of the representation theory, in terms of  which the operator $D$ looks pretty complicated.

 In the present paper  we consider arbitrary  $k$ and $m$ and use
 an equivalent language  of Stiefel manifolds.  This setting of the problem yields a dual pair of $\lam$-cosine transforms, $\C^{\lam}_{m, k}$ and $\cd0 $, which coincide when $k=m$; see (\ref{0mby}), (\ref{0mbyd}). For the dual transform $\cd0 $ we obtain a generalization of (\ref{bmuqa}) having the form
  \be \label {bmsuqa}
D_\ell \, \stackrel{*}{\Cs}\!{}_{m, k}^{\lam +2\ell}  = \, \ncd0; \qquad \ell =0,1,2, \ldots. \ee
The operator   $D_\ell$ has a simple form; see Theorem \ref  {liut0s}. 
   In the case of the unit sphere, $D_\ell$ boils down to the known polynomial of the Beltrami-Laplace operator.
We also obtain an analogue of   (\ref{bmsuqa}) for the $\lam$-sine transforms; see (\ref {sin2h}).

 Unlike \cite{AGS}, our method relies on the extension of  orthonormal Stiefel matrices by homogeneity onto the ambient space of real rectangular matrices with subsequent implementation of
 the Fourier transform technique.  This method was developed  in \cite{Ru13} and used in \cite {Ru20} to prove (\ref{bmuqa}) on the unit sphere without  spherical harmonics.  
 An analogue of (\ref{bmsuqa}) for $\C^{\lam}_{m, k}$ (without ``$\,{}^*\,$")   is still an  open problem if $k\neq m$.

 \vskip 0.2 truecm

{\bf 2.}
We also study analytic continuations of   properly normalized $\lam$-cosine transforms, which include Stiefel  analogues of the Funk transform (\ref{af000r}), as well as their intermediate modifications. We call  them   {\it the intermediate Funk-cosine transforms}  and denote  $F^{(j)}_{m,k}$; $j\!=\!0,1, \ldots, m\!-\!1$.
  In the case $j=0$, the operator  $F_{m,k}\equiv F^{(0)}_{m,k}$ is a straightforward  generalization of  (\ref{af000r}) and coincides with the latter if $m=k=1$.   
All these transforms can be written in Grassmannian terms and  expressed as convolutions with  positive  Radon measures.   Convolutions (or distributions) of  similar nature  are well known in Analysis
  and deal with integration over matrices of lower rank; see, e.g., \cite [Chapter VII, Section 2] {FK}, \cite[Section 4]{Ru06}.

Intermediate Funk-cosine transforms in the case $k=m$ were considered by
 Cross \cite{Cr}, who defined them  using the group representation tools developed by {\'O}lafsson and  Pasquale \cite{OP}. Our approach, invoking   Stiefel manifolds and zeta integrals,  is different in principle. It is straightforward, has simple geometric and group-theoretic interpretation,  and covers all admissible $k$ and $m$; see formulas (\ref{Interm}),   (\ref{rkfkfn}), (\ref{00bn9adu}), (\ref{rkfkfndu}).

 \vskip 0.2 truecm

{\bf 3.} We apply  (\ref {bmsuqa})  to  inversion of the  Funk  transforms $F_{m,k}$  and the  intermediate Funk-cosine transforms  $F^{(j)}_{m,k}$; see Section \ref{ljkb2q}.  For the sake of simplicity, the results are formulated in terms of the right $O(m)$-invariant functions on
the Stiefel manifold $\vnm$, but the reader   can  easily reformulate them  in the Grassmannian language.  We obtain new local inversion formulas  and some nonlocal formulas, the structure of which depends on the parity of dimensions and agrees with known results for the unit sphere  \cite{H11, Ru15, Ru20}.   Some cases, related to nonlocal inversion formulas, remain open and  need new ideas; see Section \ref{kjb8j} for the list of open problems that might be of interest.

It should be noted that nonlocal inversion formulas for the Funk transform on Grassmannians are known in  terms, which differ from those in the present paper.  Such formulas, invoking the Crofton symbol and the kappa operator, can be found in \cite{GGR}
 and  are pretty involved. An alternative approach in terms of G{\aa}rding-Gindikin fractional integrals
 was suggested in \cite{GR} for real Grassmannians and extended in \cite{Zh1}
to complex and quaternionic cases.  Unlike these works, our goal in the present paper is to find simple differential operators that   agree with elegant  formulas by Helgason \cite [p. 133]{H11} and  our formulas in \cite{Ru20} for the unit sphere.

Differential operators  with determinantal power weights  were studied by Sahi and Zhang \cite{SZ} in the general context of real, complex, and quaternionic matrix spaces.  These operators  have common features with
 $D_\ell$ in (\ref{bmsuqa}),
descend to Grassmannians, and can be  used to obtain local inversion formulas  for the Funk transform.  The method of \cite{SZ} heavily relies on the group representation technique and essentially  differs from ours.
In contrast with  \cite{SZ}, the core of our approach is the classical Fourier analysis. It allows us to obtain not only local but also  some nonlocal inversion formulas and
 covers intermediate Funk-cosine transforms, that were not considered in \cite{SZ}.

 \vskip 0.2 truecm

{\bf 4.}  A distinctive feature of our  paper in comparison with other related publications  (see, e.g., \cite {A, AGS, Gr1, K,  OPR, SZ}) is that we think of  smooth functions on  Stiefel (or Grassmann) manifolds not in terms of  coordinate charts, but using homogeneous continuation of the  relevant orthonormal frames onto the  space of rectangular matrices, where    classical Calculus can be applied. This transition  is performed with the aid of the polar decomposition of  matrices. As a result,  it becomes possible to write the desired differential operators in a simple analytic form.

\vskip 0.2 truecm

{\bf 5.}  The above approach to the definition of smooth functions entails,  however, some extra work. Specifically, we need to show equivalence of our definition  and the classical one, as, e.g., in the  Lie theory, and  carefully  justify   the $C^\infty \to  C^\infty$ action of  all operators under consideration\footnote{Such a justification is sometimes skipped in \cite{Ru13}, as ``obvious".}. Information about  differentiable structures on Stiefel or Grassmann manifolds and related diffeomorphisms
 is highly scattered and presented in different sources from different points of view. For  convenience of the reader, we have  written an Appendix, in which this auxiliary material is organized in a unified consistent form.
 Most of the  facts, except probably Lemma \ref{wnvp},  are well known;  some of them look folklorish.

In Sections 2-6  we introduce basic objects of our investigation and study their properties (see  Contents). The main results are  presented in Sections 7,8.

\section{Preliminaries}\label {ifoutrr}

 \subsection{Notation}\label {ioutrr}

 Let  $\frM_{n,m}$ be  the
space of real matrices $x=(x_{i,j})$ having $n$ rows and $m$
 columns. We associate $\frM_{n,m}$ with the real  space $\bbr^{nm}$ of $nm$-tuples
 \[(x_{1,1}, \ldots, x_{n,1}, x_{1,2}, \ldots, x_{n,2}, \ldots, x_{1,m}, \ldots, x_{n,m});
 \]
  $dx=\prod^{n}_{i=1}\prod^{m}_{j=1} dx_{i,j}$;
   $x'$ is  the transpose of  $x$; $|x|_m=\det
(x'x)^{1/2}$; $I_m$
   is the identity $m \times m$
  matrix; $0$ stands for the zero entries.
  In the case $n\ge m$, we denote by $\tilde\frM_{n,m}$ the set  of all matrices $x\in \frM_{n,m}$ of rank $m$. This set  is an open subset of $ \frM_{n,m}$ in the standard topology of $\bbr^{nm}$;   $GL(n,\bbr)=\tilde\frM_{n,n}$ is the general linear group of $\bbr^n$.  We write
  \[e_1=(1,0, \ldots, 0), \;e_2=(0,1, \ldots, 0),\ldots, e_n= (0, \ldots, 0,1)\]
for the coordinate unit vectors in $\rn$.

The notation $L^1(M)$, $C(M)$, $C^\infty (M)$ for the function spaces of Lebesgue integrable, continuous, and infinitely differentiable functions on $M$ is standard. It is assumed that $M$ is equipped with  suitable   structure. If a group $H$ acts on $M$ from the right, then $L^1(M)^H$, $C(M)^H$, and $C^\infty (M)^H$ denote the corresponding spaces of right $H$-invariant functions.

We will be dealing with the compact Stiefel manifold $\vnm=\{v\in \frM_{n,m} : v'v=I_m\}$ of  orthonormal $m$-frames in $\bbr^n$ and the  Grassmann manifold  $G_{n,m}$ of $m$-dimensional linear subspaces of $\bbr^n$ equipped with the relevant Haar probability measures. Basic  facts about these manifolds are collected in Appendix. If $v\in \vnm$, then $\{v\}=\span (v)\in G_{n,m}$ is a linear subspace spanned by $v$;    $v^\perp\in G_{n, n-m}$ is a subspace perpendicular  to $v$. If $m=n$, then $V_{n,n}=O(n)$ is the group of orthogonal transformations of $\rn$. If $m=1$, then  $V_{n,1}=\sn$ is the unit sphere  in $\rn$.

 The Fourier transform  of a
function $\vp\in L^1(\frM_{n,m})$ is defined by
\[
\hat\vp (y)=\intl_{\frM_{n,m}} e^{{\rm tr(iy'x)}} \vp
(x)\, dx,\qquad y\in\frM_{n,m} \; .\] The corresponding Parseval
equality
 has the form
 \be\label{parse} (\hat \vp, \hat \om)=(2\pi)^{nm} \, (\vp, \om),
\qquad (\vp, \om)=\intl_{\frM_{n,m}} \vp(x) \overline{\om(x)} \,
dx.\ee
If $\om$ belongs to the Schwartz space $S(\frM_{n,m})$ of rapidly decreasing smooth functions and $\vp\in S'(\frM_{n,m})$ is a tempered distribution, the
 equality (\ref{parse}) serves as a definition of the  Fourier transform  of $\vp$.

 The  Cayley-Laplace operator $\Del$  on  $\frM_{n,m}$  is  defined by
 \be\label{K-L} \Del=\det(\d'\d), \ee
where $\partial$ is the $n\times m$  matrix whose entries
are partial derivatives $\d/\d x_{i,j}$.  In the Fourier transform
terms, the action of $\Del$ represents a multiplication by  $(-1)^m |y|_m^2$.
It follows that $\Del$ is left $O(n)$-invariant   and right  $O(m)$-invariant, that is,
\be\label {lwliut}
\Del: \, f(\rho x) \rightarrow (\Del f)(\rho x), \qquad f(x\gam) \rightarrow (\Del f)(x\gam),\ee
for all $\rho \in O(n)$, $\gam \in O(m)$, $x \in \Ma$.
These relations can be easily checked using the Fourier transform.
More information about the  Cayley-Laplace operator can be found in \cite{Kh2, Ru06}.

 In the following, $\sym_m \simeq \bbr^{m(m+1)/2}$  is the space of $m \times m$ real
symmetric matrices $s=(s_{i,j})$; $ds=\prod_{i \le j} ds_{i,j}$;   $\Omega_m$  denotes the cone of positive definite matrices in $\sym_m$. The  Siegel gamma  function of $\Omega_m$  is defined by
\be\label{2.4}
 \gm (\a)\!=\!\intl_{\Omega_m} \det (s)^{\a-(m+1)/2 }  e^{-\tr (s)} \, ds
 =\pi^{m(m-1)/4}\prod\limits_{j=0}^{m-1} \Gam (\a\!-\! j/2);  \ee
see  \cite{FK, Ga, Gi, Si}. This integral is absolutely convergent if $Re\,\a > (m-1)/2$ and extends a meromorphic function of $\a$ with the
 polar set
 \[\{(m-1-j)/2\, : \,  j=0,1,2,\ldots\}.\]

The abbreviation $a.c.$ mean analytic continuation. Normalized probability measures will be usually denoted by $d_*$ followed by the variable of integration.  The letter $c$ (sometimes with subscripts) is used for a constant that can be different at each occurrence.

 \subsection{Zeta integrals} \label {se4}    Suppose that $n\geq m\ge 2$ and denote
 \be\label{zeta}
\Z_{n,m}(f,\lam)=\intl_{\Ma} f(x) |x|^{\lam}_m \,dx, \qquad f\in   S(\Ma),\quad \lam \in\bbc.\ee
This expression is called the  {\it zeta   integral} \cite{Ig, Shin} and  represents a meromorphic $S'$-distribution.

\begin{lemma}\label{lacz} {\rm ( \cite{Sh1, Kh1}, \cite [Lemma 4.2]{Ru06})}
 The integral (\ref{zeta})
is absolutely convergent if $Re\, \lam > m-n-1$ and  extends to $Re\, \lam \leq m-n-1$
as a meromorphic function of $ \; \lam$ with the only poles $ \;
m-n-1, m-n-2,\dots\;$. These poles and their orders are  the
same as of  $\gm((\lam +n)/2)$. The normalized
integral
\be\label{gzeta}
 \z_{n,m}(f,\lam)=\frac{\Z_{n,m}(f,\lam)}{\gm((\lam +n)/2)}\ee
 is an entire
function of $\lam$.
\end{lemma}

 If $\Del$ is the  Cayley-Laplace operator (\ref{K-L}), the following identity of the Bernstein
type holds:
 \[\Del ^\ell
|x|_m^{\lam+2\ell}=B_{\ell,m,n}(\lam)|x|_m^{\lam},\qquad \ell=1,2, \ldots \,,\]
 \be\label{bka}
B_{\ell,m,n}(\lam)=\prod\limits_{i=0}^{m-1}\prod\limits_{j=0}^{\ell-1}(\lam +n-i+2j)(\lam+2+2j+i);\ee
see \cite[p. 565]{Ru06}. It allows us to represent  meromorphic continuation  of $\Z(f,\lam)$ in the  form
\be\label {oao}
\Z_{n,m}(f,\lam)\!=\!\frac{1}{B_{\ell,m,n}(\lam)}\, \Z_{n,m}(\Del ^\ell f,\lam\!+\!2\ell), \quad  Re \,\lam > m\!-\!n\!-\!1\!-\!2\ell.\ee
The values
\[ \lam =-n,\, 1-n, \ldots, m-n-1, \]
for which the corresponding zeta distribution is a positive measure, deserve special mentioning; cf.  \cite[Theorem VII.3.1]{FK}.

\begin{lemma}\label{tzk}  \cite[Theorem 4.4, Lemma 4.7]{Ru06}  If $f  \in  S(\Ma)$, then
 \[ \z_{n,m}(f, \lam) \big |_{\lam =j-n}=\intl_{\Ma} f(x) \, d\nu_j (x),\qquad j=0,1, \ldots,  m-1,\]
 where $\nu_j$  is   a  Radon measure supported on the set $\{x \in \Ma :\, \rank (x) \le j\}$. Specifically, if $j=1,2, \ldots,  m-1$, then
 \[\z_{n,m}(f, j-n)=\frac{\pi^{(n-j)m/2}}{\gm(n/2)} \intl_{O(n)} d_*\gam
\intl_{\frM_{j,m}}f \left (\gam \left[\begin{array} {c} \om \\ 0
\end{array} \right]  \right ) \, d\om. \]
If $j=0$, then
 \[\z_{n,m}(f, -n)=\frac{\pi^{nm/2}}{\gm(n/2)}f(0). \]
\end{lemma}

\subsection{ Convolutions}

Most of the operators in our paper are expressed as convolutions on the  group $O(n)$. In general, let $G$ be a compact Lie group,   $\M (G)$  be the space of Radon  measures on $G$. The convolution of $f\in L^1 (G)$  with $\mu \in \M (G)$   is defined by
\[(f \ast \mu)(x)= \intl_G f(xy ^{-1})\,  d\mu (y) \]
and belongs to $L^1 (G)$.
\begin{proposition}\label {biid}  If $f\in C^\infty (G)$, $\mu \in \M (G)$, then  $f \ast \mu\ \in C^\infty (G)$.
\end{proposition}
\begin{proof}
This statement can be found in \cite [p. 83]{GZ} without proof.  The proof was briefly outlined in  \cite [p. 147]{H00} for the special case of Radon transforms. In the general case it can be proved as follows.

 We first note that, by definition of the Lie group, the map
\[ \varkappa:\,  G \times G  \to G,  \qquad (x,y) \to  xy^{-1},\]
is smooth. Hence the function $F= f\circ \varkappa $ is smooth on  $G \times G$, as a composition of smooth maps.  Fix any  coordinate  chart $(U, \vp)$  for  $G$ and let $B=\vp(U)$ be a Euclidean ball in $\bbr^d$, $d=\dim G$. Consider  the diffeomorphism
$ B \times G  \to U \times G$,  $\,(\xi,y) \to  (\vp^{-1} \xi,y)$,
and let $\tilde F (\xi, y)=F (\vp^{-1} \xi,y)$. If $x\in U$, the convolution $f \ast \mu$ is locally  represented as
\be\label{oaws} (f \ast \mu ) (\vp^{-1} \xi) =\intl_G \tilde F (\xi, y)\,  d\mu (y), \qquad \xi \in B.\ee
The function $\tilde F$ is smooth on $B \times G$ as a composition of smooth maps.
 Note also that $\tilde F (\xi, y)$ is smooth as a function of two variables if and only if
both  ``single-variable functions''  $\xi \to \tilde F (\xi, y)$  and   $y \to \tilde F (\xi, y)$ are smooth. In particular, if $(\xi, y)\to \tilde F(\xi, y)$ is smooth, then any partial
derivative $(\xi, y) \to \partial_\xi \tilde F (\xi, y)$ is smooth too.
The last observation allows us to differentiate under the sign of integration in (\ref{oaws}) infinitely many times, and we are done.
\end{proof}

 \section{Funk Transforms  on Stiefel and Grassmann Manifolds}\label {Funk T}

Let  $V_{n, k}$ and $V_{n, m}$ be a pair of  Stiefel manifolds; $ 1\le k,m\le n-1$.
We consider  Funk-type transforms, which are formally defined by
 \be \label {la3v}(F_{m,k} f) (u)=\intl_{\{v\in V_{n, m}: \, u'v=0\}} f(v)\,d_u v,
\qquad u\!\in\! V_{n, k},\ee
 \be \label {la3vd}(\fd \vp) (v)=\intl_{\{u\in V_{n, k}: \, u'v=0\}} \vp(u)\,d_v u,
\qquad v\!\in\! V_{n, m}.\ee
 The condition $u'v=0$ means that the  subspaces $\{u\}\in \gnk$ and $\{v\} \in \gnm$ are mutually orthogonal. Hence,    necessarily,
 \[k+m\le n.\]

To give  $(F_{m,k} f) (u)$ and $(\fd \vp) (v)$ precise meaning,  we set
 \[ u_0= \left[\begin{array} {c}  0 \\  I_{k} \end{array} \right] \in V_{n, k},
  \qquad  v_0= \left[\begin{array} {c}  I_{m} \\ 0 \end{array} \right]\in V_{n, m},
 \]
and let $ g_u$ and  $ g_v$ be orthogonal transformations satisfying $g_u u_0=u$,
 $g_v v_0=v$.  Denote $f_u (v)=f (g_u v)$, $\vp_v (u)=\vp (g_v u)$.
 Then  (\ref{la3v}) and (\ref{la3vd}) can be  explicitly written  as
 \be\label{876a}
(F_{m, k}f) (u)\!=\!\!\intl_{V_{n-k,m}} \!\!\!\!\!f_u\left(
\left[\begin{array} {c} \vartheta
\\0
\end{array} \right]\right) d_*\vartheta\!=\! \!\intl_{O(n-k)} \!\!\!\!\!f_u\left(
\left[\begin{array} {cc} \gam & 0 \\ 0 & I_k
\end{array} \right]  v_0\right)d_*\gam,
 \ee
 \be \label{876ab}(\fd\vp) (v)\!=\!\!\intl_{V_{n-m,k}} \!\!\!\!\!\vp_v\left(
\left[\begin{array} {c} 0
\\\theta
\end{array} \right]\right) d_*\theta\!= \!\! \intl_{O(n-m)} \!\!\!\!\!\vp_v\left(
\left[\begin{array} {cc} I_m& 0 \\ 0 & \rho
\end{array} \right]  u_0\right) d_*\rho.
 \ee
These expressions are independent of the afore-mentioned choice of $ g_u$ and  $ g_v$ and agree with the case $m=k=1$ of the unit sphere. Operators  $F_{m, k}$ and $\fd $ are $O(n)$-equivariant,
the
 function $F_{m, k}f$ is right $O(k)$-invariant, and $\fd \vp$  is right $O(m)$-invariant

   If  $k=m$  we set $F_m= F_{m,m}$. In this case,   (\ref{la3v}) and (\ref{la3vd}) essentially coincide.
If $k+m= n$, then $V_{n-k,m}=O(m)$, $V_{n-m,k}=O(k)$, and our Funk transforms are  averages of the form
\[(F_{m, n-m}f) (u)\!=\!\! \intl_{O(m)}\!\!\!f_u (v_0 \vartheta) d_*\vartheta, \qquad (\stackrel {*} F_{m, n-m} \vp) (v)\!=\! \!\intl_{O(k)} \!\!\!\vp_v (u_0 \theta) d_*\theta.\]

Note also that
\be\label{0zz09a}
F_{m, k} f= F_{m, k} f_{ave}, \qquad \fd   \vp =\fd   \vp_{ave},\ee
where
\[
 f_{ave} (v)=\intl_{O(m)} f (v\b)\, d_*\b, \qquad  \vp_{ave} (u)=\intl_{O(k)} \vp(u \a)\, d_*\a.\]
By Proposition \ref{lkutt},
the maps  $f \to f_{ave}$ and $\vp \to \vp_{ave}$ act from $L^1$ to $L^1$ and from $C^\infty$ to $C^\infty$ on the corresponding Stiefel manifolds. The functions  $f$, for which $f_{ave}=0$, belong to the kernel (the null space) of the operator  $F_{m, k}$   (similarly for $\fd$). Thus, in general,   $F_{m, k}$  and $\fd$  are non-injective.

The Funk transforms  $F_{m, k}f$ and $\fd \vp $ can be thought of as convolutions on the group $G=O(n)$ with delta measures $\mu_U$ and $\mu_V$ associated with  stabilizers
\[
U=\left \{g\in G:\, g= \left[\begin{array} {cc} \gam & 0 \\ 0 & I_k
\end{array} \right], \quad \gam \in O(n-k)\right \},\]
\[
V=\left \{g\in G:\, g= \left[\begin{array} {cc} I_m& 0 \\ 0 & \rho
\end{array} \right], \quad \rho \in O(n-m)\right \}\]
of $u_0$ and $v_0$, respectively. These measures are defined by
\[
\intl_G  \!\om (g)  d\mu_U (g)\!=\! \intl_U  \!\om (g) d_U g,\quad \intl_G  \!\om (g)  d\mu_V (g)\!=\! \intl_V  \!\om (g) d_V g, \quad
\om \in C(G),\]
where $d_U g$ and $d_V g$ are the relevant Haar probability measures. We denote
\[
\vp_0 (\a) =\vp (\a u_0), \qquad  f_0 (\b) =f (\b v_0); \qquad \a,\b \in G.\]
Then, by (\ref {876a}) and (\ref {876ab}),
\be\label{0ui09a}
(F_{m, k}f) (\a u_0)\!=\!\intl_G \!f_0 (\a g^{-1}) d\mu_U (g),  \ee
\be\label{0ui09ad} (\fd\vp) (\b v_0)\!=\!\intl_{G} \!\vp_0 (\b g^{-1}) d\mu_V (g). \ee

\begin{lemma} \label {aqr1}  Let $ 1\le k,m\le n-1$; $\,k+m\le n$. The operators $F_{m, k}$ and $\fd$ act  from $L^1$ to $L^1$ and from $C^\infty$ to $C^\infty$  on the corresponding Stiefel manifolds. Moreover,
\be\label{009a}\intl_{V_{n, k}}(F_{m, k}f)(u)\, \vp(u)\,
d_*u=\intl_{\vnm}f(v)\,(\fd \vp)(v)\, d_*v,\ee provided that at
least one of these integrals is finite when $f$ and $\vp$ are
replaced by $|f|$ and $|\vp|$, respectively.
\end{lemma}
\begin{proof} The first statement follows from (\ref{0ui09a}) and (\ref{0ui09ad}), taking into account the  properties of convolutions on compact Lie groups (use, e.g.,  Propositions \ref{thmVNM}  and   \ref{biid}).
  The duality  (\ref{009a}) agrees with Helgason's double fibration scheme \cite[p. 144]{H00}. A straightforward proof of (\ref{009a}) can be found in  \cite [Lemma 3.2] {Ru13}; see also (\ref{009adu}) for the  more general statement.
\end{proof}

 There is an obvious relationship between the  Funk transforms  (\ref{la3v}) and (\ref{la3vd}) and
Radon type transforms on Grassmannians, defined by
 \be \label {gla3vd}
(R_{p,q}  \check f)(\eta)=\intl_{\xi
\subset \eta} \!\! \check f(\xi)\, d_\eta \xi, \qquad (\stackrel{*}{R}_{p,q}  \vp_{\!\circ})(\xi)= \intl_{\eta \supset \xi}   \vp_{\!\circ}(\eta) \, d_\xi \eta,\ee
\[
\xi \in G_{n,p}, \qquad \eta \in G_{n,q}, \qquad 1\le p\le q\le n-1,\]
  $d_\eta \xi$ and $ d_\xi \eta$ being the relevant probability measures.
Specifically, suppose that $f$ is a right $O(m)$-invariant function on $\vnm$, $\vp$ is a right $O(k)$-invariant  function on $\vnk$, and set $p=m$, $q=n-k$. If we define   $\check f$ on $G_{n,m}$  and   $\vp_{\!\circ}$ on   $G_{n,n-k}$ by
 \be \label {glpod}
\check f(\{v\})=f(v),  \qquad \vp_{\!\circ} (u^\perp) =\vp (u), \ee
then
 \be\label {ioyv}(F_{m, k}f)(u)\!=\! (R_{m,n-k} \check f)(u^\perp), \quad (\fd \vp)(v)\!= \! (\stackrel{*}{R}_{m,n-k} \vp_{\!\circ})(\{v\}).\ee
In the case $p=q$, both expressions in (\ref{gla3vd}) represent the identity maps.

\begin{lemma} \label {paqr1} $1\le p\le q\le n-1$. The operators $R_{p,q}$ and $\stackrel{*}{R}_{p,q}$ act  from $L^1$ to $L^1$ and from $C^\infty$ to $C^\infty$  on the corresponding Grassmannians. Moreover,
\be\label{009agr} \intl_{G_{n,q}}(R_{p,q}  \check f)(\eta)\, \vp_{\!\circ}(\eta)\,
d_*\eta=\intl_{G_{n,p}} \check f(\xi)\, (\stackrel{*}{R}_{p,q}  \vp_{\!\circ})(\xi)\, d_*\xi,\ee
provided that at
least one of these integrals is finite when $\check f$ and $\vp_{\!\circ}$ are
replaced by $|\check f|$ and $|\vp_{\!\circ}|$, respectively.
\end{lemma}

This well known statement follows from Lemma \ref {aqr1} and Remark \ref{kjutrgr}. Just as Lemma \ref {aqr1}, it also  falls into the scope of Helgason's double fibration theory.

\section{The $\lam$-Cosine  Transforms}\label{t0mby}

\noindent
In this section we follow our paper \cite{Ru13}, however, the notation for some parameters   has been changed for the sake of consistency  with \cite{Ru15, Ru20}.

\subsection{Preparations}   Let $1\le m, k\le n-1$. The non-normalized $\lam$-cosine transform   and its dual are defined by
\bea
\label{0mby}(\C^{\lam}_{m, k} f)(u)&=&\intl_{\vnm} \!\!\!f(v)\,
|u'v|_m^{\lam} \, d_*v, \qquad u\in \vnk,\\
\label{0mbyd}(\cd0 \vp)(v)&=&\intl_{\vnk} \!\!\vp(u)\, |u'v|_m^{\lam} \, d_*u,\qquad v\in \vnm.
\eea
   In the self-adjoint case  $m=k$ we set
\[(\C^{\lam}_{m} f)(u)=(\C^{\lam}_{m, m} f)(u)\equiv\intl_{\vnm} \!\!\!f(v)\, |\det (u'v)|^{\lam}
\, d_*v.\]
Recall that $|u'v|_m= \det (v'uu'v)^{1/2}$, where $v'uu'v$ is a positive semi-definite $m\times m$ matrix.
We restrict our consideration to $m\le k$, because, otherwise, $|u'v|_m=0$ for all $v\in \vnm$ and  $u\in\vnk$.

The functions  $\C^{\lam}_{m, k} f$ and    $\cd0 \vp$ are  a right $O(k)$-invariant and right $O(m)$-invariant, respectively. Moreover,
\be\label {iou6v1}
 \C^{\lam}_{m, k} f = \C^{\lam}_{m, k} f_{ave}, \qquad \cd0 \vp = \cd0 \vp_{ave}, \ee
as in (\ref{0zz09a}).

Because the quantity $|u'v|_m$ is invariant under  change of variables $u \to u\a$, $\a \in O(k)$, and $v \to v\b$,  $\b \in O(m)$, it is actually a function of Grassmannian variables
\[\xi=\{v\} \in G_{n,m} \quad \text {\rm  and}  \quad \t=\{u\} \in G_{n,k}. \]
We denote this function by  $|\Cos (\xi, \t)|$, taking into account that  if $k=m=1$, then $|u'v|_m$  is exactly the cosine of the smallest angle between the lines $\xi$ and $\t$.  Thus we define
\be\label{u78udf}
|\Cos (\xi, \t)|\equiv |\Cos (\{v\}, \{u\})| = |u'v|_m.\ee
This definition does not depend on the choice of the orthonormal bases $v$ in $\xi$ and $u$ in $\t$.  Geometrically, $|\Cos (\xi, \t)|$ is the $m$-volume of the orthogonal projection onto $\t$ of a generic set of unit volume in $\xi$.

Setting
\[
\check f(\xi)\equiv \check f (\{v\})=f(v), \qquad \check \vp(\t)\equiv \check \vp (\{u\})=\vp (u),\]
\be\label{u7zf}  (T^{\lam}_{m, k} \check f)(\t)\equiv (T^{\lam}_{m, k} \check f)(\{u\})   =(\C^{\lam}_{m, k} f)(u),\ee
\be\label{u7zfa}
 (\cdt \check \vp)(\xi)\equiv (\cdt \check \vp)(\{v\}) =(\cd0 \vp)(v),\ee
we can write  (\ref{0mby}) and (\ref{0mbyd}) in the Grassmannian language as
\bea\label{0mrrby}(T^{\lam}_{m, k} \check f)(\t)&=&\intl_{G_{n,m}} \!\!\!\check f(\xi)\, |\Cos (\xi, \t)|^\lam d_* \xi, \qquad \t \in G_{n,k}, \\
\label{0mrrbydf}  (\cdt \check \vp)(\xi)&=& \intl_{G_{n,k}} \!\!\check \vp(\t)\,  |\Cos (\xi, \t)|^\lam \, d_*\t,\qquad \xi \in G_{n,m},\eea
and reformulate all our results in these terms. However,  for the sake of convenience (especially in proofs), we prefer the Stiefel terminology.

Note that, unlike (\ref{0mby}) and (\ref{0mbyd})\footnote {Although $GL(n,\bbr)$ does not act directly on $\vnm$, one can consider representations of this group
on the spaces of right $O(m)$-invariant functions on $\vnm$; see \cite{OPR}, \cite [Section 7.4.3]{Ru13},  and references therein.}, the operators $T^{\lam}_{m, k}$ and $\cdt$ are $GL(n,\bbr)$-equivariant, because
\be\label{0cvbydf}
|\Cos (g\xi, g\t)| =   |\Cos (\xi, \t)|\quad \text{\rm for all} \quad g\in GL(n,\bbr).\ee
The latter can be easily checked if we write $g$ in polar coordinates $g=\om r^{1/2}$, where $\om \in O(n)$ and $r=g'g$ is a positive definite $n\times n$ matrix. Specifically, if $\xi=\{v\}$, then
\bea
g\xi&=&g\{v\}=\{gv\}=\{\om r^{1/2} v\}\nonumber\\
&=&\om  \{r^{1/2} v\}= \om  \{v\}= \{\om v\}\nonumber\eea
(similarly for $g\t$). Hence
\bea
|\Cos (g\xi, g\t)|&=& |\Cos (\{\om v\}, \{\om u\})| =  |(\om u)'(\om v)|_m \nonumber\\
&=&|u'v|_m =  |\Cos (\xi, \t)|.\nonumber\eea

\begin{lemma} \label{lhgn1} Let  $1\le m\le k\le  n-1$. If  $\Re \lam>m-k-1$, then the operators $\C^{\lam}_{m, k}$,  $\cd0$, $T^{\lam}_{m, k}$, and  $\cdt$ act  from $L^1$ to $L^1$ and from $C^\infty$ to $C^\infty$  on the corresponding Stiefel or Grassmann manifolds.
\end{lemma}
\begin{proof}  We can write $(\C^{\lam}_{m, k} f)(u)$ and $(\cd0 \vp)(v)$ as  convolutions on the group $G=O(n)$.  Specifically, let
\[
 u_0\!= \!\left[\begin{array} {c}  0 \\  I_{k} \end{array} \right]\in V_{n, k}, \quad  v_0\!=\! \left[\begin{array} {c}  I_{m} \\  0 \end{array} \right]\in V_{n, m}; \qquad u=\a u_0,  \quad v=\b v_0.\]
Setting $ f_0(\b)= f(\b v_0)$,   $ \vp_0(\a)= \vp(\a u_0)$, we obtain
\[
(\C^{\lam}_{m, k} f)(\a u_0)=\intl_G  f_0(\a\gam^{-1})\, h(\gam)\, d_*\gam,\]
\[ (\cd0\vp)(\b v_0)=\intl_G  \vp_0(\b\gam^{-1}) \,h^*(\gam)\, d_*\gam,\]
where
$h(\gam)=|u'_0 \gam' v_0|_m^{\lam}$,  $\;h^*(\gam)=|u'_0 \gam v_0|_m^{\lam}$  (recall that $\gam^{-1}=\gam'$). If $\Re \lam>m-k-1$, then  $h$ and $h^*$ are  integrable on $G$. The latter follows from the equality
\bea\label{ddsil}
\intl_G h(\gam) \, d_*\gam \!\!&=&\!\!\!\intl_G h^*(\gam) \, d_*\gam\nonumber\\
\label{ddsil}\!\!&=&\!\!\!\intl_{\vnm} |u'_0 v|_m^{\lam}\, d_*v=\frac{\Gam_{m} (n/2)\, \Gam_{m} ((\lam +k)/2)}
{\Gam_{m} (k/2)\,\Gam_{m}((\lam+n)/2)};\eea
see \cite  [formula (A.16)]{Ru13}.  Now the $L^1 $ action is obvious and the smoothness result  holds by Proposition \ref{biid}.  The corresponding statements for  $T^{\lam}_{m, k}$   and  $\cdt$ then follow from (\ref {u7zf})
and   (\ref {u7zfa})  by Remark \ref{kjutrgr} and Proposition \ref{muu1}.
  \end{proof}

   \begin{remark}\label {reem1} {\rm The condition $\Re \lam>m-k-1$  in Lemma \ref{lhgn1} is sharp because $(\C^{\lam}_{m, k}1)(u)\!=\!(\cd0 1)(v)$ coincide with (\ref{ddsil}). If $Re\,\lam\le m-k-1$, then the gamma function
 $\Gam_{m} ((\lam +k)/2)$ in this expression represents  a divergent integral.}
\end{remark}

\subsection{Connection Between the $\lam$-Cosine Transform and Its Dual}

Given $u \in V_{n,k}$ and $v \in V_{n,m}$, we  denote by $\tilde u \in V_{n, n-k}$ and $\tilde v\in V_{n, n-m} $  arbitrary frames,  which are orthogonal to the subspaces $ \{u\}=\span (u)$ and $ \{v\}=\span (v)$, respectively. By Proposition \ref{mnxmnb1} there is  a one-to-one  correspondence
$  f   \simeq f_*$  between   the right $O(m)$-invariant  functions $f$ on $\vnm$ and right $O(n-m)$-invariant
 functions $f_*$ on $V_{n, n-m}$  (similarly $  \vp   \simeq \vp_*$). Clearly,
 \[
  \intl_{V_{n,k}} \!\vp (u)\,d_*u= \!\!\intl_{V_{n,n-k}} \vp_* (\tilde u)\,d_*\tilde u,\qquad    \intl_{V_{n,m}} \!f (v)\,d_*v= \!\!\intl_{V_{n,n-m}} f_* (\tilde v)\,d_*\tilde v,\]
  which follows from (\ref{wgted}).

\begin{proposition} \label {mnxmnb} Let  $1\le m\le k\le  n-1$,  $\Re \lam>m-k-1$. If $\vp \in L^1 (V_{n,k})^{O(k)}$, then
\be\label {k5gg80}
(\cd0\vp)(v)= (\C^{\lam}_{n-k, n-m}  \vp_*)(\tilde v ).\ee
\end {proposition}
\begin{proof}






Note that
   $\det (v'uu'v)= \det (\tilde u'\tilde v \tilde v'  \tilde u)$.
  The latter can be proved using Sylvester's equality
  \[  \det (I_m - ab) = \det (I_n - ba); \qquad  a\in \frM_{m,n}, \; b\in \frM_{n,m}.\]
 Indeed,
\[
 \det (v'uu'v)=\det (I_m-v'\tilde u \tilde u'v)=\det (I_{n-k}-\tilde u'vv'\tilde u )=\det (\tilde u'\tilde v \tilde v'  \tilde u ). \]
Hence, by (\ref{wgted}),
\bea
(\cd0\vp)(v)&=&\intl_{\vnk} \!\!\vp(u)\, |u'v|_m^{\lam} \, d_*u\nonumber\\
&=&\intl_{ V_{n, n-k}}  \vp_* (\tilde u) \, |\tilde v'  \tilde u |_{n-k}^{\lam} \, d_*\tilde u=(\C^{\lam}_{n-k, n-m}  \vp_*)(\tilde v ). \nonumber\eea
\end {proof}

\subsection{Analytic Continuation}   Given a frame $u\in \vnk$, we denote by $g_u$  an orthogonal transformation that takes $u_0= \left[\begin{array} {c}  0 \\  I_{k} \end{array} \right]\in\! V_{n,k}$ to $u$ and set $f_u (v)=f (g_u v)$, $v\in \vnm$.

\begin{theorem}\label{lxhgn1}
Let  $1\le m\le k\le  n-1$. If  $f\in C^\infty (\vnm)$, then  the function
\[ \lam \mapsto (\C^{\lam}_{m,k} f)(u), \qquad \Re \lam > m-k-1,\]
 extends meromorphically to  $\Re \lam \leq m-k-1$. The polar set of the extended function consists of the poles
 $ \;
m-k-1, m-k-2,\dots\;$ of  $\gm((\lam +k)/2)$.
 The normalized
integral
\[ I_f (\lam,u)=\frac{(\C^{\lam}_{m,k} f)(u)}{\gm((\lam +k)/2)}\]
 is an entire
function of $\lam$. Moreover, if
\be\label {kug5ez}
m-n \le j-k \le m-k -1, \qquad j \ge 0,\ee
then
\bea\label{09kys7x}
I_f (j-k,u) &\equiv& \underset
{\lam =j -k}{a.c.}  I_f (\lam,u)\\
&=& c_j  \intl_{O(k)} d\gam \intl_{V_{n-k+j,m}}\!\!\! f_u \left(  \left[\begin{array} {cc} I_{n-k} & 0 \\  0 & \gam \end{array} \right] \left[\begin{array} {c}  \om \\  0 \end{array} \right]\right)\,d_*\om,\nonumber\eea
\[
 c_j = \frac{\gm(n/2)}{\gm(k/2)\, \gm((n-k+j)/2)}.\]
 In particular, if $j=0$, $n-k\ge m$, then
\be
\underset {\lam = -k}{a.c.} \, \frac{(\C^{\lam}_{m,k} f)(u)}{\gm((\lam +k)/2)} \!=\! c_0\, (F_{m, k}f) (u),  \quad  c_0 \!=\! \frac{\gm(n/2)}{\gm(k/2)\, \gm((n\!-\!k)/2)},\ee
 where $F_{m, k}f$ is the Funk transform  (\ref{876a}). If $j -k< m-n$, then $I_f (j-k,u) \equiv 0$.
\end{theorem}
\begin{proof}
The reasoning below is a generalization of \cite [Subsection 7.1]{Ru13}.
By invariance, it suffices to assume $u=u_0= \left[\begin{array} {c}  0 \\  I_{k} \end{array} \right]$.  Let
\be\label{0bg54bys}
F (\lam) =\intl_{\frM_{n,m}} f (x (x'x)^{-1/2})\, |u'_0x|_m^{\lam}\, \psi (x'x)\,  e^{-\tr (x'x)}\,dx,
\ee
where $\psi $ is a nonnegative $C^\infty$ function on the cone $\Om_m$  (see Notation)
 with compact support away from the boundary of $\Om_m$.
 The function
 \be\label {j7n9} \vp  (x)\equiv f (x (x'x)^{-1/2})\, \psi (x'x)\,  e^{-\tr (x'x)}\ee
 belongs to $S  (\frM_{n,m})$ and is supported away from the surface $\det (x'x)=0$.    Passing to polar coordinates
 $x=wr^{1/2}$, $w\in \vnm$, $r\in \Om_m$   (see Lemma \ref{l2.3}), we obtain
  \[F (\lam) =\varkappa (\lam)\, (\C^{\lam}_{m,k} f)(u_0), \]
 \[\varkappa  (\lam)=2^{-m} \sig_{n,m}\intl_{\Om_m} \det(r)^{(\lam+n-m-1)/2} \psi (r)\,  e^{-\tr (r)}\, dr.\]
 Because $\varkappa (\lam)$ and its reciprocal are entire functions, the  analyticity  of $\lam \to (\C^{\lam}_{m,k} f)(u_0)$ is equivalent to that of $F (\lam)$  and  the poles of both functions  are  the same and   have the same order.

The integral  (\ref{0bg54bys}) can be represented as
\[ F (\lam)=\intl_{\frM_{n,m}}\vp (x) \, |u'_0x|_m^{\lam}\,dx=\intl_{\frM_{k,m}}\tilde\vp (y)\,|y|_m^{\lam}\,dy= \Z_{k,m}(\tilde\vp, \lam)\]
(cf. (\ref{zeta})), where the function
\be \label {45d} \tilde\vp (y)=\intl_{\frM_{n-k,m}}
\vp  \left(\left[\begin{array} {c}  \eta \\  y \end{array} \right]\right)\, d \eta \ee
belongs to  $S  (\frM_{k,m})$.  Thus,
\be \label {45cxd}
(\C^{\lam}_{m,k} f)(u_0)=\varkappa (\lam)^{-1}\,\Z_{k,m}(\tilde\vp, \lam)\ee
 and
\be \label {45xxad}
 I_f (\lam,u_0)=\frac{\Z_{k,m}(\tilde\vp, \lam)}{\varkappa (\lam)\,\gm((\lam +k)/2)}= \frac{\z_{k,m}(\tilde\vp, \lam)}{\varkappa (\lam)}; \ee
cf. (\ref{gzeta}).

 Analytic properties of  $\Z_{k,m}(\tilde\vp, \lam)$ and $\z_{k,m}(\tilde\vp, \lam)$ are described in Lemmas \ref{lacz}  and \ref{tzk} (with $n$ replaced by $k$). In particular, the integral
 $\Z_{k,m}(\tilde\vp, \lam)$ converges absolutely if $Re\, \lam > m-k-1$ and
 extends to $Re\, \lam \leq m-k-1$
as a meromorphic function of $ \; \lam$. The polar set of the extended function is a subset  of the set of poles
 $ \;
m-k-1, m-k-2,\dots\;$ of  $\gm((\lam +k)/2)$. The normalized integral $\z_{k,m}(\tilde\vp, \lam)$ is an entire
function of $\lam$.  Since the left-hand side of (\ref{45xxad}) is independent of the choice of $\psi$, the analytic continuation of the right-hand side (in which $\psi$ is hidden) is independent  of $\psi$ too, thanks to the uniqueness property of analytic functions.

If $0\le j\le m-1$,
then for $\lam = j-k \le m-k -1$ we have
 \be\label{09kys} \z_{k,m}(\tilde\vp, j-k) =\frac{\pi^{(k-j)m/2}}{\Gam_m(k/2)} \intl_{O(k)} d\gam \intl_{\frM_{j,m}} \tilde \vp  \left(\gam \left[\begin{array} {c}  \om \\  0 \end{array} \right]\right)\, d \om. \ee
In particular,
\be\label{09kys1}
\z_{k,m}(\tilde\vp, -k) =\frac{\pi^{km/2}}{\Gam_m(k/2)} \tilde \vp (0).\ee
Combining (\ref{09kys}) and (\ref{09kys1}) with (\ref{45d}) and (\ref{j7n9}),  we obtain
\[
I_f (j\!-\!k,u_0)\!=\!\frac{\pi^{(k-j)m/2}}{\Gam_m(k/2)\, \varkappa (j\!-\!k)} \intl_{O(k)}\! d\gam \!\intl_{\frM_{n-k+j,m}} \!\!\!\!\!\vp  \left(\tilde \gam \left[\begin{array} {c}  \xi \\  0 \end{array} \right]\right) d \xi, \]
\[
\tilde \gam= \left[\begin{array} {cc} I_{n-k} & 0 \\  0 & \gam \end{array} \right] \in O(n),\]
\[
\vp  \left(\tilde \gam \left[\begin{array} {c}  \xi \\  0 \end{array} \right]\right)=  f  \left(\tilde \gam \left[\begin{array} {c}  \xi \\  0 \end{array} \right]  (\xi' \xi)^{-1/2}   \right)\, \psi (\xi' \xi)
 e^{-\tr (\xi' \xi)}.\]
If $n-k+j\ge m$, that is, $m-n \le j-k$   (cf. (\ref{kug5ez})),
 then, passing to polar coordinates in $\frM_{n-k+j,m}$ (see Lemma \ref{l2.3}), we get  $I_f (j-k,u_0)=c\,I_1 \,I_2$,
 where
 $$
I_1\!=\! \intl_{O(k)} d\gam \intl_{V_{n-k+j,m}}\!\!\! f \left( \tilde \gam \left[\begin{array} {c}  \om \\  0 \end{array} \right]\right)\,d_*\om, $$
$$
I_2\!=\!\intl_{\Om_m} \! \det (r)^{(n-k+j -m-1)/2} \psi (r)\,  e^{-\tr (r)}\, dr,
$$
$$
c=\frac{\pi^{(k-j)m/2}\, \sigma_{n-k+j,m}}{\sigma_{n,m}\,\gm(k/2) \,I_2}=\frac{\gm(n/2)}{\gm(k/2)\, \gm((n-k+j)/2)}\,\frac{1}{I_2}\, .
$$
Hence,  if $\max (m+k-n, 0)\le j \le m-1$, then
\[
I_f (j-k,u_0)= c_j  \intl_{O(k)} d\gam \intl_{V_{n-k+j,m}}\!\!\! f \left(  \left[\begin{array} {cc} I_{n-k} & 0 \\  0 & \gam \end{array} \right] \left[\begin{array} {c}  \om \\  0 \end{array} \right]\right)\,d_*\om,\]
\[
 c_j = \frac{\gm(n/2)}{\gm(k/2)\, \gm((n-k+j)/2)}.\]
 If $j=0$ and $n-k \ge m$, the expression for  $I_f (j-k,u)$ has a simpler form
 \[
I_f (-k,u_0)\!=\! c_0 \!\! \intl_{V_{n-k,m}}\!\!\! f \left(\left[\begin{array} {c}  \om \\  0 \end{array} \right]\right) d_*\om= c_0\, (F_{m, k}f) (u_0);\]
cf. (\ref{876a}). If $j -k< m-n$, then the rank of $\xi$ is less than $m$, $\xi'\xi$ is the boundary point of the cone $\Om_m$,   and therefore, by the definition of  $\vp$ in (\ref{j7n9}), we have $\vp\equiv 0$. This gives $I_f (j-k,u_0)\equiv 0$.
\end{proof}

\begin{remark}\label {mbmku}
The function $I_f (\lam,u)=(\C^{\lam}_{m,k} f)(u)/\gm((\lam +k)/2)$ may have zeros at some $\lam$. For example, if $f\equiv 1$, then, by (\ref{ddsil}),
\[
I_f (\lam,u)=\frac{\Gam_{m} (n/2)}{\Gam_{m} (k/2)\,\Gam_{m}((\lam+n)/2)}=0 \quad \forall \, \lam=m-n-1, m-n-2, \ldots\, .\]
Thus $ \C^{\lam}_{m,k} f$ and $\gm((\lam +k)/2)$ may have poles of different order.
\footnote { The converse statement in
 \cite[Theorem 7.1(i)]{Ru13} should be corrected.}
\end{remark}

\begin{lemma}  \label {jygc} Let  $1\le m\le k\le  n-1$. If $f\in C^\infty (\vnm)$, then the function
\be \label {a45cxdz}
 u \to a.c.\, \frac{(\C^{\lam}_{m,k} f)(u)}{\gm((\lam +k)/2)}, \qquad u \in \vnk,\ee
is infinitely differentiable  for every complex $\lam$.
\end{lemma}
\begin{proof} Let first $\Re \lam > m-k-1$. We replace $f$ in (\ref{45cxd})  by $f_\gam = f\circ \gam$, $\gam \in G=O(n)$, to get
\[(\C^{\lam}_{m,k} f)(\gam u_0)=\varkappa (\lam)^{-1}\,\Z_{k,m}(\tilde\vp_\gam, \lam),\]
\[ \tilde\vp_\gam (y)\!=\!\!\intl_{\frM_{n-k,m}}\!\! \vp_\gam  \left(\left[\begin{array} {c}  \eta \\  y \end{array} \right]\right) d \eta, \qquad  y\in \frM_{k,m},\]
\[ \vp_\gam  (x)\!=\!\om (x) f(\gam v)\big |_{v=x (x'x)^{-1/2}},\quad
\om (x)\!=\! \psi (x'x)\,  e^{-\tr (x'x)} \!\in C^\infty (\frM_{n,m}),\]
where $\varkappa (\lam)$ and its reciprocal are entire functions of $\lam$; cf. (\ref{j7n9}), (\ref{45d}).
The maps
\[w : G \times \vnm \to \vnm, \qquad (\gam, v) \to \gam v, \]
\[ \sig : \tilde \frM_{n,m} \to \vnm, \qquad  x \to x (x'x)^{-1/2},\]
 are smooth; see Proposition \ref{mmxzc}    and the proof of Lemma \ref{wnvp}. Hence the function
\bea
F(\gam, x)&\equiv& \vp_\gam  (x)=\om (x) f(\gam v)\big |_{v=x (x'x)^{-1/2}}\nonumber\\
&=&\om (x) (f\circ w)(\gam, \sig (x))\nonumber\eea
is  smooth  on $G\times \tilde \frM_{n,m}$, and therefore
\[
F_1(\gam, \eta, y)=\vp_\gam  \left(\left[\begin{array} {c}  \eta \\  y \end{array} \right]\right)=F\left (\gam, \left[\begin{array} {c}  \eta \\  y \end{array} \right]\right)\]
is a smooth function on $G \times \tilde \frM_{n-k,m} \times \tilde \frM_{k,m}$. It follows that $ \tilde\vp_\gam (y)$ is a smooth function of $(\gam, y)\in G \times \tilde \frM_{k,m}$.
Using the meromorphic continuation formula (\ref{oao}) for zeta integrals, we obtain
\bea
a.c. \,(\C^{\lam}_{m, k} f)(\gam u_0)&=&\frac{1}{\varkappa (\lam)\, B_{\ell,m,k}(\lam)}\, \Z_{k,m}(\Del ^\ell \tilde\vp_\gam,\lam +2\ell)\nonumber\\
 &=&\frac{1}{\varkappa (\lam)\, B_{\ell,m,k}(\lam)} \intl_{\frM_{k,m}}\!\! (\Del ^\ell \tilde\vp_\gam)(y)\, |y|_m^{\lam +2\ell}\, dy,\qquad \nonumber\eea
 $$
 Re \,\lam > m-k-1-2\ell,\qquad  \ell=1,2, \ldots, \, ,$$
where $B_{\ell,m,k}(\lam)$ is the Bernstein  polynomial (\ref{bka}) (with $n$ replaced by $k$).
By above,  $\Del ^\ell \tilde\vp_\gam (y)$ is a smooth function of  $(\gam, y)$, and therefore, $a.c. \,(\C^{\lam}_{m, k} f)(\gam u_0)$ is a smooth function of $\gam \in G$.
Now the smoothness of the normalized function (\ref{a45cxdz}) follows from  Proposition \ref{thmVNM}.
\end{proof}

The next statement is an analogue of Theorem \ref{lxhgn1} and Lemma \ref{jygc} for the dual  transform  $\cd0 \vp$.

\begin{theorem}\label{lxhgnf} Let  $1\le m\le k\le  n-1$.

\vskip 0.2 truecm

\noindent {\rm (i)} If  $\vp \in C^\infty (\vnk)$, then  the function
$ \lam \mapsto (\cd0 \vp)(v)$
 extends  meromorphically to   $\Re \lam \leq m-k-1$ for every $v \!\in\! V_{n,m}$.   The polar set of the extended function consists of the poles
 $ \;
m-k-1, m-k-2,\dots\;$ of $\Gam_{n-k}((\lam +n-m)/2)$.

\vskip 0.2 truecm

\noindent {\rm (ii)}  The normalized function
 \[\lam \to \frac{(\cd0 \vp)(v)}{\Gam_{n-k}((\lam +n-m)/2)}\]
 is an entire
function of $\lam$  belonging to  $C^\infty (\vnm)^{O(m)}$ in the $v$-variable.

\vskip 0.2 truecm

\noindent {\rm (iii)} An alternative normalized function
\[ \lam \to \frac{(\cd0 \vp)(v)}{\gm((\lam +k)/2)}\]
 extends
meromorphically    with the only possible poles
\[
-k-1, \,-k -2,\dots \,. \]
\end{theorem}
\begin{proof}
By (\ref {iou6v1}), it  suffices  to consider right $O(k)$-invariant functions $\vp$. The statements
(i) and (ii) follow from  Theorem \ref{lxhgn1}, Propositions \ref{jygc}, \ref{mnxmnb}, and \ref{mnxmnb1}. To prove {\rm (iii)}, we observe that by (\ref{2.4}),
\be \Gam_{n-k}((\lam+n-m)/2)= c (\lam)\, \gm((\lam +k)/2)), \ee
where $c (\lam)= \pi^{(n-k-m)m/2}\Gam_{n-k-m}((\lam+n-m)/2)$ is a meromorphic function with the polar set $ \{
-k-1, \,-k-2,\dots \}$.
\end{proof}

\section{Intermediate  Funk-Cosine Transforms} \label{Intermediate}

 Theorem \ref{lxhgn1} leads to new Radon-like transforms
 \be\label{Interm}
 (F_{m,k}^{(j)} f) (u)=  \intl_{O(k)} d_*\gam \intl_{V_{n-k+j,m}}\!\!\!
 f_u \left(  \left[\begin{array} {cc} I_{n-k} & 0 \\  0 & \gam \end{array} \right] \left[\begin{array} {c}  \om \\  0 \end{array} \right]\right)\,d_*\om,\ee
  which take functions on $\vnm$ to functions on $\vnk$. Following this theorem, we assume
 \be\label{cerm} 1\le m\le k\le  n-1, \qquad n-k+j \ge m, \qquad 0\le j \le m-1.\ee
 Recall that
  \[f_u (v)\!=\! f (g_u v), \quad g_u \in G=O(n), \quad g_u u_0 =u, \quad u_0\!=\!\left[\begin{array} {c}  0 \\  I_k \end{array} \right] \in \vnk.\]
  One can   formally write (\ref{Interm}) as
\be\label{Intermra}
 (F_{m,k}^{(j)} f) (u)= \intl_{\{v\in \vnm : \;  {\rm rank} (u' v)\le j\}} \!\!\!\!\! f(v)\,  d_{u,j} (v). \ee
  The case $j=0$ agrees with the usual Funk transform $F_{m,k}$.
   By (\ref{09kys7x}),
\be \label {kfkfn}
\underset {\lam = j-k}{a.c.} \, \frac{(\C^{\lam}_{m,k} f)(u)}{\gm((\lam +k)/2)} \!=\! c_j\, (F_{m,k}^{(j)} f) (u),  \ee
\be\label{cnnerm}
 c_j = \frac{\gm(n/2)}{\gm(k/2)\, \gm((n-k+j)/2)}.\ee


The integral transform (\ref{Interm}) has a nice geometric interpretation in the Grassmannian language  (\ref{gla3vd}). Specifically, suppose first $u=u_0$.
Given a right $O(m)$-invariant function $f$ on $\vnm$, we define the associated function $\check f$ on $G_{n,m}$   by   $\check f (\{v\})=f(v)$. Denote
\[\eta_0 \!=\!\span ( e_{1}, \ldots , e_{n-k})\!= \!\bbr^{n-k}, \quad \z_0\!=\!\span ( e_1, \ldots , e_{n-k+j})\!= \!\bbr^{n-k+j}. \]
 Then the inner integral in (\ref{Interm}) can be written as
\be\label{cneem}
\intl_{\{\xi \in G_{n,m}: \; \xi\subset \tilde \gam\z_0\}} \!\!\!\! \check f (\xi)\, d_{\tilde \gam} \xi\!=\!(R_{m, n-k+j} \check f)(\tilde \gam\z_0), \qquad \tilde \gam \!=\!\left[\begin{array} {cc} I_{n-k} & 0 \\  0 & \gam \end{array} \right].\ee
 Let
\[
 G_{n, n-k+j} (\eta)=\{\z \in G_{n, n-k+j}: \z \supset \eta\}, \qquad  \eta \in G_{n, n-k}.\]
Integrating (\ref{cneem}) over $\gam \in O(k)$, and noting that $\tilde \gam$ leaves $\eta_0$ fixed,
 we obtain
\bea
 (F_{m,k}^{(j)} f) (u_0) &=&\intl_{ G_{n, n-k+j} (\eta_0)}\!\!\!\! (R_{m, n-k+j} \check f)(\z)\, d_*\z\nonumber\\
\label{0yydu}  &=&  (\stackrel {*}R_{n-k, n-k+j} R_{m, n-k+j} \check f)(\eta_0).\eea
 Hence, by rotation invariance,
\be\label{0uudu}
 (F_{m,k}^{(j)} f) (u)= (\stackrel {*}R_{n-k, n-k+j} R_{m, n-k+j} \check f)(\eta), \qquad \eta=u^\perp.\ee
We denote
\be \label {rkfkfn}
 (R_{m,n-k}^{(j)} \check f) (\eta)=(\stackrel {*}R_{n-k, n-k+j} R_{m, n-k+j} \check f)(\eta), \qquad \eta \in G_{n, n-k}.\ee
This expression is the Grassmannian model of the intermediate  Funk-cosine transform  $F_{m,k}^{(j)}$.
If $j=0$, it  boils down to the usual Radon transform $R_{m, n-k} \check f$, as in (\ref{gla3vd}).

The {\it dual intermediate  Funk-cosine transform $\stackrel{*}{F}{}_{m,k}^{\!(j)}$} is naturally defined as an integral operator
satisfying
\be\label{009adu}\intl_{V_{n, k}}(F_{m,k}^{(j)} f) (u)\, \vp(u)\,
d_*u=\intl_{\vnm}f(v)\,(\stackrel{*}{F}{}_{m,k}^{\!(j)} \vp) (v)\, d_*v.\ee

To obtain  an explicit formula for $\stackrel{*}{F}{}_{m,k}^{\!(j)} \vp$, we set
\[u_0\!=\!\left[\begin{array} {c}  0 \\ I_{k} \end{array} \right], \qquad v_0\!=\!  \left[\begin{array} {c}  I_{m} \\ 0 \end{array} \right], \]
 \[
\tilde \gam= \left[\begin{array} {cc} I_{n-k} & 0 \\  0 & \gam \end{array} \right],\quad  \tilde a= \left[\begin{array} {cc} a & 0 \\  0 &  I_{k-j} \end{array} \right], \quad
\tilde b= \left[\begin{array} {cc} I_{m} & 0 \\  0 &  b \end{array} \right],\]
where $\gam \in O(k)$, $\,a \in O(n-k+j)$, $\,b \in O(n-m)$.
Then, by (\ref{Interm}),
 \bea
 (F_{m,k}^{(j)} f) (u)&=&\intl_{O(k)} d_*\gam \intl_{O(n-k+j)}\!\!\! f_u (\tilde \gam  \tilde a  v_0)\, d_*a\nonumber\\
 &=&\intl_{O(n-m)} d_*b \intl_{O(k)} d_*\gam \intl_{O(n-k+j)}\!\!\! f_u (\tilde \gam  \tilde a \tilde b v_0)\, d_*a.\nonumber\eea
 Hence
 \bea
I&=& \intl_{V_{n, k}}(F_{m,k}^{(j)} f) (u)\, \vp(u)\,d_*u=\intl_{O(n)} (F_{m,k}^{(j)} f) (gu_0)\, \vp(gu_0)\,d_*g  \nonumber\\
 &=&\intl_{O(n)} \vp(gu_0)\,d_*g\intl_{O(n-m)} d_*b \intl_{O(k)} d_*\gam \intl_{O(n-k+j)}\!\!\! f (g\tilde \gam  \tilde a  \tilde b v_0)\, d_*a \nonumber\\
 &=&\intl_{O(n)} f(\lam v_0) d_*\lam \intl_{O(n-m)} d_*b \intl_{O(k)} d_*\gam \intl_{O(n-k+j)}\!\!\! \vp (\lam  \tilde b  \tilde a \tilde \gam u_0)\, d_*a. \nonumber\eea
Thus we can set
\[
(\stackrel{*}{F}{}_{m,k}^{\!(j)} \vp)(v)= \intl_{O(n-m)} d_*b \intl_{O(n-k+j)} d_*a \intl_{O(k)}\!\!\! \vp_v (\tilde b  \tilde a \tilde \gam u_0)\, d_*\gam,\]
where  $\vp_v (u)\!=\! \vp (g_v u)$, $\,g_v \in O(n)$,  $\,g_v v_0 =v$.   If  $\vp$ is right $O(k)$-invariant, then
\be\label{00bn9adu}
(\stackrel{*}{F}{}_{m,k}^{\!(j)} \vp)(v)= \intl_{O(n-m)} d_*b \intl_{O(n-k+j)}  \!\!\! \vp_v (\tilde b \tilde a  u_0)\, d_*a\ee
(note  that $\tilde \gam u_0=u_0\gam$).
If $j=0$, the above formula gives  the usual dual Funk transform (\ref{la3vd}).

To obtain a Grassmannian analogue of (\ref{00bn9adu}),  we define a function  $\vp_{\!\circ}$ on $G_{n,n-k}$   by the formula  $\vp_{\!\circ} (u^\perp) =\vp (u)$, $u\in \vnk$   ( cf. (\ref {glpod})),  and set
 $\xi_0 =\{v_0\}$. Then, as in (\ref{0yydu}),
\[
(\stackrel{*}{F}{}_{m,k}^{\!(j)} \vp)(v_0)=  (\stackrel{*}{R}{}_{m,n-k+j} R_{n-k, n-k+j} \vp_{\!\circ})(\xi_0),\]
and, by rotation invariance,
\be \label {rakfkfndu}
(\stackrel{*}{F}{}_{m,k}^{\!(j)} \vp)(v)=  (\stackrel{*}{R}{}_{m,n-k+j} R_{n-k, n-k+j} \vp_{\!\circ})(\{v\}).\ee
Thus the Grassmannian modification of the dual intermediate  Funk-cosine transform is
\be \label {rkfkfndu}
(\stackrel{*}{R}{}_{m,n-k}^{\!(j)} \vp_{\!\circ}) (\xi)= (\stackrel{*}{R}{}_{m,n-k+j} R_{n-k, n-k+j} \vp_{\!\circ})(\xi), \qquad \xi\in G_{n, m}.\ee

\begin{lemma}  \label {jygcs} If  $j, k, m, n$ satisfy (\ref{cerm}), then the operators
 $F_{m,k}^{(j)}$, $\stackrel{*}{F}{\!}_{\!m,k}^{\!(j)}$, $R{}_{m,n-k}^{\!(j)}$, and $\stackrel{*}{R}{}_{m,n-k}^{\!(j)}$
 act  from $L^1$ to $L^1$ and from $C^\infty$ to $C^\infty $ on the corresponding Stiefel or Grassmann manifolds.
 \end{lemma}
\begin{proof}  The result for $R{}_{m,n-k}^{\!(j)}$ and $\stackrel{*}{R}{}_{m,n-k}^{\!(j)}$ follows immediately from the composition formulas (\ref {rkfkfn}) and  (\ref {rkfkfndu}) according to
Lemma \ref{paqr1}.  By (\ref {0uudu}) and (\ref {rakfkfndu}), these operators are expressed through  $F_{m,k}^{(j)}$ and $\stackrel{*}{F}{\!}_{\!m,k}^{\!(j)}$. Hence the result for $F_{m,k}^{(j)}$ and $\stackrel{*}{F}{\!}_{\!m,k}^{\!(j)}$  follows from Remark \ref{kjutrgr} and Proposition \ref{muu1}.
 Alternatively,  the result for arbitrary $L^1$ or smooth functions on the Stiefel  manifolds can be obtained if we represent our operators as convolutions   with Radon measures  on $O(n)$, as we did  in the proof of Lemma \ref{aqr1}.
\end{proof}

\begin{proposition} \label{n0mnby} Let  $1\le m\le k\le  n-1$, $\vp \in C^\infty (\vnk)$.
If $j\ge 0$ satisfies $ m-n \le j-k\le m-k-1$,  then
  \be\label{zn0xdr} \underset
{\lam=j-k}{a.c.} \,\frac{ (\cd0 \vp)(v)}{\gm((\lam +k)/2)}=c_j\,(\stackrel{*}{F}{}_{m,k}^{\!(j)} \vp)(v),\ee
$ c_j $  being the constant (\ref {cnnerm}).
\end{proposition}
\begin{proof} Denote
$$
A_\lam (v)\equiv\frac{ (\cd0 \vp)(v)}{\gm((\lam +k)/2))}=\frac{c (\lam)\, (\cd0 \vp)(v)}{\Gam_{n-k}((\lam+n-m)/2)}, \quad Re\, \lam >m-k-1.$$
By  Theorem \ref{lxhgnf} {\rm (ii)},
 this function extends   analytically to $Re\, \lam>-k-1$ and the analytic continuation  belongs to  $ C^\infty (\vnm)$. Clearly, $j-k >-k-1$, because $j\ge 0$. Hence,
 for any test function $w \in C^\infty (\vnm)$, owing to (\ref{kfkfn}) and (\ref{009adu}), we have
\bea
(\underset {\lam=j-k}{a.c.} \,A_\lam, w)&=& \underset {\lam=j-k}{a.c.} \,\left (\vp, \frac{\C^{\lam}_{m, k} w}{\gm((\lam +k)/2))}\right )\nonumber\\
&=&\left (\vp,  \underset {\lam=j-k}{a.c.} \,\frac{\C^{\lam}_{m, k} w}{\gm((\lam +k)/2))}\right )\nonumber\\
&=&c_j\,(\vp, F_{m, k}^{\!(j)} w)=c_j\,(\stackrel{*}{F}{}_{m,k}^{\!(j)} \vp,  w),\nonumber\eea
and (\ref{zn0xdr}) follows.

\end{proof}

\section{Normalized $\lam$-Cosine and $\lam$-Sine Transforms}\label{ncnfr}

\subsection { Normalized $\lam$-Cosine Transforms}

Let  $1\!\le\! m\le\! k\!\le\! n\!-\!1$. We introduce the following normalized modifications of the cosine transforms (\ref{0mby}) and  (\ref{0mbyd}):
\be \label{n0mby}(\Cs_{m,k}^\lam f)(u)=\gam_{m,k} (\lam)\intl_{\vnm} \!\!\!f(v)\,
|u'v|_m^\lam \, d_*v, \quad u\in \vnk,\ee
\be \label{n0mbyd}(\ncd0 \vp)(v)=\gam_{m,k} (\lam) \intl_{\vnk} \!\!\vp(u)\, |u'v|_m^\lam  \, d_*u, \quad v\in \vnm;\ee
\[\gam_{m,k} (\lam)=\frac{\Gam_m(m/2)}{\Gam_m(n/2)}\,\frac{\Gam_m(-\lam/2)}{\Gam_m((\lam +k)/2)}, \qquad \lam \neq 1-m,\, 2-m, \ldots \, .
\]
Such a normalization makes our operators consistent with those in the case $m=1$   (cf. \cite{Ru15, Ru20})
and simplifies many formulas in the sequel.  Excluded values of $\lam$ belong to the polar set  of  $\Gam_m(-\lam/2)$.
Both integrals exist in the Lebesgue sense if $\Re \lam>m-k-1$.
If $k=m$,  we set $\Cs_{m}^\lam f=\Cs_{m,m}^\lam f$.

\begin{theorem}\label {mnbqw} Let $f\in C^\infty (\vnm)$,  $1\le m\le k\le n-1$. \hfill

\vskip 0.2 truecm

\noindent {\rm (i)} The function  $\lam \to \Cs_{m,k}^\lam f$ extends  meromorphically  with the only poles  $\lam \!=\!1-m,\, 2-m, \ldots$. The extended function belongs to  $C^\infty (\vnk)^{O(k)}$.

\vskip 0.2 truecm

\noindent {\rm (ii)} If
\be\label {liddyz} m-n\le j-k \le \min (-m, m-k-1), \qquad  j\ge 0,\ee
 then
\be\label {kfskfn}
\underset {\lam = j-k}{a.c.} \,\Cs_{m,k}^\lam f = \tilde c_j\, F_{m,k}^{(j)} f,  \ee
\be\label {kfrskfn}
\tilde  c_j = \frac{\gm(m/2)\, \gm((k-j)/2)}{\gm(k/2)\, \gm((n-k+j)/2)}.\ee
In particular, for $j=0$,  $m+k\le n$,
\[\underset {\lam = -k}{a.c.} \,\Cs_{m,k}^\lam f = \tilde c_0\, F_{m,k} f,  \qquad \tilde  c_0 = \frac{\gm(m/2)}{\gm((n-k)/2)}.\]
\end{theorem}
\begin{proof} The statement (i) follows from Theorem \ref{lxhgn1} and Proposition \ref{jygc}. The restrictions  $ m-n\le j-k \le  m-k-1$, $j \ge 0$, are inherited from (\ref{kug5ez}). The inequality  $j-k \le  -m$ means that $\lam=j-k$ does not belong to the polar set $\{1-m,\, 2-m, \ldots\}$. The statement  (ii) follows from (\ref{kfkfn}).
\end{proof}

\begin{remark}\label {mniibqw} The inequality  $ m-n\le j-k \le - m$ in (\ref{liddyz}) implies $2m\le n$. The additional restriction $j-k \le - m$  is not imposed in the definition (\ref{Interm}). The case
\be\label {kwfskfn}
m-n \le 1-m \le j-k \le m-k-1,\ee
when  $F_{m,k}^{(j)} f$  is well defined, but the left-hand side of (\ref{kfskfn}) may be infinite, is not included in Theorem \ref{mnbqw}.
\end {remark}

\begin{conjecture} \label {m0opqw} In the case (\ref {kwfskfn}), the operator  $F_{m,k}^{(j)}$ is non-injective on $C^\infty (\vnm)^{O(m)}$.
\end{conjecture}

For the dual transform $\stackrel{*}{\Cs}\!{}_{m, k}^{\lam} \vp$, the following result is a consequence of Theorem \ref{lxhgnf} and Proposition \ref {n0mnby}.

\begin{theorem}\label {mnbqwo} Let $\vp\in C^\infty (\vnk)$,
\[1\le m\le k\le  n-1, \qquad  n-k+j \ge m, \qquad 0\le j \le m-1.\]

\noindent {\rm (i)} The function  $\lam \to \stackrel{*}{\Cs}\!{}_{m, k}^{\lam} \vp$ extends  meromorphically  with the only poles
\[\lam \in \{-k-1, -k-2, \ldots \} \cup \{ 1-m,\, 2-m, \ldots  \}.\]
The extended function belongs to  $C^\infty (\vnm)^{O(m)}$.

\vskip 0.2 truecm

\noindent {\rm (ii)}  If, moreover,   $j-k \le \min (-m, m-k-1)$, then
\be\label{zn0xdro} \underset
{\lam=j-k}{a.c.} \, \stackrel{*}{\Cs}\!{}_{m, k}^{\lam} \vp=\tilde c_j\,\stackrel{*}{F}{}_{m,k}^{\!(j)} \vp,\ee
where $\tilde c_j$ is defined by (\ref {kfrskfn}).
\noindent In particular, for $j=0$,  $m+k\le n$,
\[\underset {\lam = -k}{a.c.} \, \stackrel{*}{\Cs}\!{}_{m, k}^{\lam} \vp = \tilde c_0\, \fd \vp,  \qquad \tilde  c_0 = \frac{\gm(m/2)}{\gm((n-k)/2)}.\]
\end{theorem}

\subsection {Normalized $\lam$-Sine Transforms} \label{ty0mby}

The {\it normalized $\lam$-sine transform} is defined by
\be\label{tag1.10aach5stn}
(\S_m^\lam f)(u)\!=\!  \del_m (\lam) \!\intl_{\vnm} \!\!\det (I_m\!-\! v'uu'v)^{\lam/2} f(v) \,d_*v , \quad u \in \vnm,\ee
\[
 \del_m (\lam)=\frac{\Gam_m(m/2)}{\Gam_m(n/2)}\,\frac{\Gam_m(-\lam/2)}{\Gam_m((\lam +n-m)/2)}, \quad 2m\le n; \quad \lam+m\neq 1,2, \ldots\,.\]
 More general sine transforms acting from $\vnm$ to $\vnk$ were introduced in \cite [Sections 4,6]{Ru13}.   If $f \in L^1 (\vnm)$, the integral  (\ref{tag1.10aach5stn}) is absolutely convergent provided $Re \,\lam > 2m-1-n$.\footnote{Here and on, when dealing with the normalized $\lam$-sine transforms, we use the formula (6.6) from \cite{Ru13}, in which $\a+m-n$ should be replaced by $\lam$. Similarly, when dealing with the normalized $\lam$-cosine transforms, we use formulas (6.1)-(6.2) from \cite{Ru13} in which $\a-k$ is replaced by $\lam$. }

 Note that
 \be\label{tyy51so}
 \intl_{\vnm} \!\!\det (I_m\!-\! v'uu'v)^{\lam/2}  \,d_*v=\frac{\Gam_m(n/2)\, \Gam_m((\lam +n-m)/2)}{\Gam_m((n-m)/2)\,\Gam_m((\lam +n)/2)}\ee
  (cf. \cite [Remark 4.4]{Ru13}). This formula shows that the restriction $Re \,\lam > 2m-1-n$ is sharp.

 The function $\S_m^\lam f$ is right $O(m)$-invariant and
 \[\S_m^\lam f = \S_m^\lam f_{ave}, \qquad f_{ave} (v)=\intl_{O(m)} f (v\gam)\, d_*\gam,\]
for every  $f \in L^1 (\vnm)$.
 Hence, in many occurrences, when dealing with $\S_m^\lam f$, it suffices to assume $f$ to be  right $O(m)$-invariant.

 If $\Pr_{\{u\}}$ and $\Pr_{\{u\}^\perp }$ stand for the orthogonal projections onto the  subspaces  $\{u\}$ and $\{u\}^\perp$, then
 \bea
\det (I_m\!-\! v'uu'v)&=&\det (I_m -v'\Pr_{\{u\}}v)=\det (v' \Pr_{\{u\}^\perp } v)\nonumber\\
&=&\det (v' \tilde u \tilde u' v), \nonumber\eea
 where $\tilde u $ is an  $(n-m)$-frame  orthogonal to $ \{u\}$. It follows that
 \be\label{atyy51s}
(\S_{m}^\lam f)(u)= (\Cs_{m,n-m}^\lam f)(\tilde u), \qquad 2m\le n. \ee
  The assumption $2m\le n$ is natural because otherwise,
 \[\det (I_m\!-\! v'uu'v)=\det (v' \tilde u \tilde u' v)=0.\]

The equality (\ref{atyy51s}) combined with  Proposition \ref{mnxmnb1} and Theorem \ref{mnbqw} yields the following result (see \cite [Theorem 7.2]{Ru13} for details).
\begin{lemma} \label{lcjr1}  If  $f\in C^\infty (\vnm)^{O(m)}$, then for each  $u\in \vnm$, $(\S_{m}^\lam f)(u)$ extends meromorphically to all complex $\lam$ with the only poles
\[\lam = 1-m, \,2-m,  \ldots\, ,\]
 so that $a.c. \,(\S_{m}^\lam f)(u) \in C^\infty (\vnm)^{O(m)}$ in the $u$-variable.
 Moreover,
\be\label{tyy5hh}
\underset
{\lam=m-n}{a.c.} \,(\S_{m}^\lam f)(u) =f(u), \qquad 2m \le n.\ee
\end{lemma}

The following statement establishes remarkable connection between the sine transforms,  cosine transforms, and  Funk transforms.

\begin{lemma} \label{lcjrw} {\rm (}cf. \cite [Theorems 4.5, 4.8]{Ru13}{\rm )} Let  $1\le m\le k \le n-m$, $f\in L^1 (\vnm)$. If
\[Re \,\lam > 2m-1-n, \qquad \lam\neq 1-m, \,2-m, \ldots\, ,\]
  then
\be\label{ktyy51}
\S_{m}^\lam f\!=\!\tilde \del  \ncd0 F_{m,k} f\!=\!\tilde \del   \fd\Cs_{m,k}^\lam f,  \quad \tilde \del =\frac{\Gam_{m} (k/2)}{\Gam_{m} ((n\!-\!m)/2)}.\ee
In particular, if  $\, 2m \le n-k$, then
\be\label{tyy51b}
\S_{m}^{-k} f\!=\!  \del_0   \fd F_{m,k} f, \quad \del_0 \!=\!\frac{\Gam_{m}(k/2)\, \Gam_{m}
(m/2)}{\Gam_{m}((n\!-\!k)/2)\, \Gam_{m}((n\!-\!m)/2)}.\ee
 \end{lemma}

The proof of these formulas is actually an application of Fubini's theorem.

Lemma \ref{lcjr1} implies the following

\begin{corollary} \label {mcmndfi} If $f \in C^\infty (\vnm)^{O(m)}$, $1\le m\le k\le n-m$, then (\ref{ktyy51}) extends to all
 complex $ \lam\neq 1-m, \,2-m, \ldots\,$. In particular, analytic continuations of $\ncd0 F_{m,k} f$ and  $\fd\Cs_{m,k}^\lam f$ belong to $ C^\infty (\vnm)^{O(m)}$.
\end{corollary}

The next result extends (\ref{tyy51b}) to intermediate Funk-cosine transforms.
\begin{corollary} \label {pvvlz}  If $f \in C^\infty (\vnm)^{O(m)}$, $1\le m\le k\le n-m$,
\be\label {liddyzs} m-n\le j-k \le \min (-m, m-k-1), \qquad  j\ge 0,\ee
then
\be\label{tyy51}
\S_{m}^{j-k} f\!=\!  \del_j \stackrel{*}{F}{}_{m,k}^{\!(j)} F_{m,k} f= \del_j   \fd F_{m,k}^{(j)} f, \ee
\[
  \del_j = \frac{\gm((k-j)/2)\, \gm(m/2)}{\gm((n-k+j)/2)\,\Gam_{m}((n\!-\!m)/2)}\, .\]
\end{corollary}
 This statement  follows from (\ref{kfskfn}), (\ref{zn0xdro}), and (\ref{ktyy51}).  The assumption (\ref{liddyzs}) mimics  those in Theorems \ref{mnbqw} and \ref{mnbqwo}.

\section{The  Fourier Transform and Differential Operators}


 Now, after we are done with all preparations,  we can proceed to the main topic of the paper.
  We first introduce an auxiliary
  integral operator
\be \label {akm}(A_{k,m} \vp)(v)\!=\!\intl
_{V_{n-m, k-m}}\!\!\!\!\!\vp \left (g_v \left[\begin{array} {cc} a & 0 \\  0 & I_{m} \end{array} \right]\right )\, d_*a,  \qquad v\!\in \!\vnm,
\ee
where $1\le m\le k\le n-1$,    $g_v \left[\begin{array} {c}  0 \\  I_{m} \end{array} \right]=v$, $ g_v\in O(n)$.
   If $k=m$,  $A_{k,m}$ is the identity operator. One can show \cite [Lemma 5.2]{Ru13} that $A_{k,m}$ is a linear bounded operator from $L^1 (\vnk)$ to $L^1 (\vnm)$.
Given a function $f$  on $\vnm$, using polar decomposition (\ref{pol}), we set
\be\label {lamm} (E_\lam f)(x)=|x|_m^\lam f(x (x'x)^{-1/2}), \qquad x\in \tilde\frM_{n,m}.\ee

\begin{theorem}\label{th:33} \cite [Corollary 5.5]{Ru13}
Let $\,\vp\in L^1 (\vnk)^{O(k)}$,  $\om\in
 S(\Ma)$, $1\!\le\! m\! \le \!k\le\! n\!-\!1$. Then for every $ \lam\in\bbc$,
 \be\label{eq5ya}
 \left(\frac{E_{\lam} \cd0 \vp}{\Gam_m ((\lam +k)/2)}, \hat \om\right)=c \, \left (\frac{E_{-\lam-n} A_{k,m} \vp}{\Gam_m (-\lam/2)},  \om\right),\ee
\[ c =\frac{2^{m(n+\lam)}\, \pi^{nm/2}\,\Gam_m (n/2)}{\Gam_m (m/2)},\]
where both sides   are understood in the sense of analytic continuation.
\end{theorem}

 For the normalized transform $\ncd0 \vp$, (\ref{eq5ya}) yields
\be\label {liyz}
 (E_{\lam} \!\ncd0 \vp, \hat\om)=c_{m,\lam} \, (E_{-\lam-n} A_{k,m} \vp,  \om),\ee
 \[c_{m,\lam}\!=\!2^{m(n+\lam)}\, \pi^{nm/2},\qquad \lam \neq 1-m, \, 2-m, \ldots \,.\]

We define the following  differential operator on $\vnm$:
\be\label {lase1m}
(\Delta_{\lam} f) (v)\!=\! \left (-\frac{1}{4}\right)^m\,(\Del  E_{\lam +2} f)(x)\big |_{x=v}, \ee
\[ v\in \vnm, \qquad  \lam\in\bbc,\]
 where $\Del$ is the  Cayley-Laplace operator (\ref{K-L}).
If $m=1$, then   $\Delta$ is the usual Laplace operator and $\Delta_{\lam}$ expresses through  the  Beltrami-Laplace operator  $\Delta_S$  on the unit sphere as
\[
\Delta_{\lam} f=-\frac{1}{4}\,[\Delta_S f+ (\lam +2)(n+\lam) f].\]
The latter can be easily checked using the product formula
\[ \Del (\vp \psi)= \vp\Del \psi +2\,(\grad \,\vp )\cdot (\grad \,\psi) +\psi\Del \vp;\]
cf. \cite[Proposition 2.2]{Ru20}.

\begin{lemma}\label {wmb}
If $f$ is a right $O(m)$-invariant function on $\vnm$, then $\Delta_{\lam} f$ is a right $O(m)$-invariant function on $\vnm$ too, and
therefore $\Delta_{\lam} $ can be viewed  as a differential operator on the Grassmannian $G_{n,m}$.
\end{lemma}
\begin{proof}
 We need to show that $(\Delta_{\lam} f) (v\gam)=(\Delta_{\lam} f) (v)$ for all $\gam \in O(m)$, $v\in \vnm$. Let
\[F (x)=|x|_m^{\lam +2} f(x (x'x)^{-1/2}), \qquad x \in \tilde \frM_{n,m}.\]
Then
\[
(\Delta_{\lam} f) (v\gam )\!=\! \left (-\frac{1}{4}\right)^m\,(\Del  E_{\lam +2} f)(x)\big |_{x=v\gam}\!=\! \left (-\frac{1}{4}\right)^m\,(\Del F)(x)\big |_{x=v\gam}.\]
 By (\ref{lwliut}),
\[
(\Del F)(x)\big |_{x=v\gam}\equiv (\Del F)(x\gam)\big |_{x=v} = (\Del F_\gam)(x)\big |_{x=v}.\]
Here
$F_\gam(x)=F(x\gam)=|x\gam|_m^{\lam +2} f(x\gam (\gam'x'x\gam)^{-1/2})$. Using  polar decomposition $x=vr^{1/2}$, $v\in \vnm$, $r=x'x\in \Om_m$, we write
\[
 f(x\gam (\gam'x'x\gam)^{-1/2})=f(vr^{1/2} \gam (\gam'r\gam)^{-1/2}).\]
The $m\times m$ matrix $r^{1/2} \gam$ can be written in the polar form as $r^{1/2} \gam=\th s^{1/2}$, $\th \in O(m)$,  $s\in \Om_m$. Hence we continue:
\[
 f(vr^{1/2} \gam (\gam'r\gam)^{-1/2})=f(v\th s^{1/2} s^{-1/2})=f(v\th)= f(v),\]
because $f$ is right $O(m)$-invariant. Noting that $|x\gam|_m = |x|_m$, we obtain
$F(x\gam)=|x|_m^{\lam +2} f(x (x'x)^{-1/2})$, which means that $(\Delta_{\lam} f) (v\gam)=(\Delta_{\lam} f) (v)$.
\end{proof}

An analogue of Lemma \ref{wmb} still holds   if we replace the Cayley-Laplace operator $\Del$ by its power and set
\be\label {lase1hh}
(\Delta_{\lam, \ell} f) (v)= \left (-\frac{1}{4}\right )^{m \ell} (\Del^\ell E_{\lam +2\ell} f)(x)\Big |_{x=v}; \qquad \ell = 1,2, \ldots \, . \ee

\begin{theorem}\label {liut0s}  Let   $1\le m\le k\le n-1$,   $\vp \in C^\infty (\vnk)^{O(k)}$.
Then for all complex $\lam$,
 \be\label {lase2hh}
\Delta_{\lam, \ell} \stackrel{*}{\Cs}\!{}_{m, k}^{\lam +2\ell} \vp=\ncd0 \vp,\ee
provided that both sides of this  equality are meaningful and smooth.
     \end{theorem}
 \begin{proof}   Let first $\ell =1$. Suppose
 $\om \in S(\frM_{n,m})$, and let  $ \hat\om(y)$
  be the Fourier transform  of $\om$.  Then, by (\ref{liyz}),
      \be\label{eq5sa}  (E_{\lam} \ncd0 \vp, \hat\om)\!=\!c_{m,\lam} \, (E_{-\lam-n} A_{k,m} \vp,  \om),\quad c_{m,\lam}\!=\!2^{m(n+\lam)}\, \pi^{nm/2},\ee
for all complex $\lam\notin \{1\!-\!m, \,2\!-\!m, \ldots \}\,$.
 Setting $\om_1 (x)=|x|_m^{2} \om (x)$ and
using (\ref{eq5sa}) repeatedly, we obtain
\bea
\left (\Del [E_{\lam +2} \! \stackrel{*}{\Cs}\!{}_{m, k}^{\lam +2} \vp], \hat\om\right)&=& (-1)^m\left (E_{\lam +2} \!  \stackrel{*}{\Cs}\!{}_{m, k}^{\lam +2} \vp, \hat\om_1 \right)\nonumber\\
&=& (-1)^m c_{m,\lam +2} \, \left (E_{-\lam-2-n} A_{k,m} \vp,  \om_1\right)\nonumber\\&=&  (-1)^m c_{m,\lam +2} \,\left (E_{-\lam-n} A_{k,m} \vp,  \om\right)\nonumber\\
&=&(-1)^m \frac{c_{m,\lam +2}}{c_{m,\lam}} \left(E_{\lam} \!\ncd0 \vp, \hat\om\right)\nonumber\\&=& (-4)^m \left(E_{\lam}\! \ncd0 \vp, \hat\om\right).  \nonumber\eea
  Since  $\stackrel{*}{\Cs}\!{}_{m, k}^{\lam +2} \vp$ and $\ncd0 \vp$ are smooth on $\vnm$, the functions
 \[
 \Del [E_{\lam +2} \! \stackrel{*}{\Cs}\!{}_{m, k}^{\lam +2} \vp] \quad \text {\rm and}\quad  E_{\lam}\! \ncd0 \vp\]
   are smooth on $\tilde \frM_{n,m}$, and therefore
 \[ (\Del [E_{\lam +2} \! \stackrel{*}{\Cs}\!{}_{m, k}^{\lam +2} \vp]) (x)=(-4)^m  (E_{\lam}\! \ncd0 \vp)(x)\]
  for all $x\in \tilde \frM_{n,m}$. Setting $x=v\in \vnm$, we obtain the result.
 In the general case the proof is similar: just  set $\om_1 (x)=|x|_m^{2\ell} \,\om (x)$) to obtain
\[ (\Del^\ell E_{\lam +2\ell} \stackrel{*}{\Cs}\!{}_{m, k}^{\lam +2\ell} \vp) (x)=(-4)^{m \ell}\,  (E_{\lam}\, \ncd0 \vp)(x), \quad x\in \tilde \frM_{n,m}. \]
This completes the proof.
  \end{proof}

 \begin{remark} \label {lkug} By Theorem \ref{mnbqwo} (i), the conditions of Theorem \ref{liut0s} are satisfied  if
  \[\lam, \lam +2\ell  \notin \{-k-1, -k-2, \ldots \} \cup \{ 1-m,\, 2-m, \ldots  \}.\]
 If, moreover, $ k\le n-m$, then, by  Corollary \ref{mcmndfi},  both $\stackrel{*}{\Cs}\!{}_{m, k}^{\lam +2\ell} \vp$ and $\ncd0 \vp$ are smooth for $\lam +2\ell \neq 1\!-\!m, \,2\!-\!m, \ldots$
 provided that  $\vp$ lies in the range of the Funk transform, i.e.,  $\vp =F_{m,k} f$,  $f \in C^\infty (\vnm)$.
 \end{remark}

Combining the formula $\S_{m}^\lam f\!=\!\tilde \del  \ncd0 F_{m,k} f$ (see (\ref{ktyy51}))
with  (\ref{lase2hh}), we obtain
 \be\label {sin2h}
\Delta_{\lam, \ell}\, \S_m^{\lam +2\ell} f=\S_m^{\lam} f,  \qquad \lam \in \bbc, \quad \lam +2\ell\neq 1\!-\!m, \,2\!-\!m, \ldots,\ee
where $1\le m\le n-m$ and $\Delta_{\lam, \ell}$ is defined by  (\ref{lase1hh}).
Regarding applicability of this reasoning,  we recall that whenever $1\le m\le k\le n-m$,  $F_{m,k}$ acts from $C^\infty (\vnm)$ to $C^\infty (\vnk)^{O(k)}$  and
$\stackrel{*}{\Cs}\!{}_{m, k}^{\lam}$ acts from the range $F_{m,k} (C^\infty (\vnm))$ to  $C^\infty (\vnm)^{O(m)}$ for all
$\lam \in \bbc$, except poles.

\section{Inversion  Formulas}\label {ljkb2q}

As usual in the Radon transform theory, we distinguish the local  inversion formulas and the nonlocal ones, depending on the parity of dimensions involved. In the case of the Funk-cosine transform $F_{m,k}^{(j)}$, the formulas of the fist kind correspond to  $n-m+j-k$  even, while the  most difficult case, when $n-m+j-k$ is odd, deals with formulas of the second kind. For the sake of simplicity and consistency with the previous text,
all inversion formulas in this section are presented for functions on  the Stiefel manifold $\vnm$. Because functions on the Grassmannian $G_{n,m}$ can be viewed as   right $O(m)$-invariant functions on $\vnm$, the
 reader can easily reformulate the results in the Grassmannian terms.

\subsection{Local Inversion}

 \begin{theorem} \label{jhb67}  Let  $\vp = F_{m,k}^{(j)} f$, $\,f\in C^\infty (\vnm)^{O(m)}$,
\[ 1\le m\le k\le  n-m, \qquad m-n\le j-k \le \min (-m, m-k-1), \quad  j\ge 0.\]
If $n-m+j-k$ is even and $\ell =(n-m+j-k)/2\ge 0$, then
\[f= \del_j \, \Delta_{m-n, \ell} \,\psi, \qquad \psi= \fd \vp,  \]
where
\[  \del_j = \frac{\gm((k-j)/2)\, \gm(m/2)}{\gm((n-k+j)/2)\,\Gam_{m}((n\!-\!m)/2)}, \]
\[
(\Delta_{m-n, \ell}\, \psi) (v)= \left (-\frac{1}{4}\right )^{m \ell} (\Del^\ell E_{m-n +2\ell} \,\psi)(x)\Big |_{x=v}, \qquad v \in \vnm. \]
\end{theorem}
\begin{proof}
 By (\ref{tyy5hh}),  $f= \S_{m}^{m-n} f$, where,  by
(\ref{sin2h})  with $\lam=m-n$, we have
\[\S_{m}^{m-n} f=\Delta_{m-n, \ell}\, \S_m^{j-k} f.\]
Because $\S_m^{j-k} f=\del_j \fd F_{m,k}^{(j)} f$ (see (\ref{tyy51})), the result follows.
\end{proof}

The case $j=0$, when  $\vp= F_{m,k} f$ is the usual Funk transform, deserves special mentioning.  If $n-m-k$ is even, the above theorem yields
\be\label {entgce}
f= \del_0 \, \Delta_{m-n, \ell}  \fd \vp,  \ee
where
\be\label {errnce}
  \del_0 \!=\! \frac{\gm(k/2)\, \gm(m/2)}{\gm((n\!-\!k)/2)\,\Gam_{m}((n\!-\!m)/2)}, \qquad \ell \!=\!\frac{n-m-k}{2}\ge 0.\ee

In the case $m=k=1$,  (\ref{entgce}) agrees with Helgason's formula in \cite [Theorem 1.17, p. 133]{H11} and the corresponding formula in \cite{Ru20}. In the general case, our formula can be used as a substitute for known local inversion formulas in
\cite{GGR, Gr1, GR, K, SZ}.

\subsection{Nonlocal  Inversion of $F_{m,k}$  in the Case  $m<k$}

If $n-k-m$ is odd, a local inversion formula, like (\ref{entgce}), is not available. However, if $ m< k$,  the following  result for the Funk transform $F_{m,k}$ (i.e., $j=0$) can be obtained  by making use of the intermediate dual Funk-cosine  transform $ \stackrel{*}{F}{}_{m,k}^{\!(1)}$.

 \begin{theorem} \label{jhbsa67} Let $\vp \!= \! F_{m,k} f$, $f \!\in C^\infty (\vnm)^{O(m)}$,   $1\le m< k\le  n \! -  \!m$. Suppose that $n-k-m$ is odd. Then
\[f= c\, \Delta_{m-n, \ell}\,\stackrel{*}{F}{}_{m,k}^{\!(1)} \vp,  \]
where
\[c =\frac{\gm(m/2)\, \gm((k-1)/2)}{\gm((n-m)/2)\, \gm((n-k+1)/2)}, \qquad \ell =\frac{n-k-m+1}{2}.\]
  \end{theorem}
\begin{proof} By  (\ref{ktyy51}) and Corollary \ref{mcmndfi} (with $\lam =1-k$) we have
\[\S_{m}^{1-k} f\!=\!\tilde \del  \stackrel{*}{\Cs}\!{}_{m, k}^{1-k} \vp, \qquad \tilde \del =\frac{\Gam_{m} (k/2)}{\Gam_{m} ((n\!-\!m)/2)},\]
 where, by (\ref{zn0xdro}),
\[\stackrel{*}{\Cs}\!{}_{m, k}^{1-k} \vp=\tilde c_1\,\stackrel{*}{F}{}_{m,k}^{\!(1)} \vp, \qquad \tilde  c_1 = \frac{\gm(m/2)\, \gm((k-1)/2)}{\gm(k/2)\, \gm((n-k+1)/2)}.\]
Hence, by (\ref{sin2h}),
\bea
f&=&\S_{m}^{m-n} f=\Delta_{m-n, \ell}\, \S_m^{m-n+2\ell} f=\Delta_{m-n, \ell}\, \S_m^{1-k} f \nonumber\\
&=& \tilde \del \Delta_{m-n, \ell}\,\stackrel{*}{\Cs}\!{}_{m, k}^{1-k} \vp=
\tilde \del \tilde c_1\, \Delta_{m-n, \ell}\,\stackrel{*}{F}{}_{m,k}^{\!(1)} \vp,\nonumber\eea
which gives the result.
\end{proof}

\subsection{The case  $k=m$, $j=0$}

In this case  the dual Funk transform intertwines the differential operator in the local inversion formula,  and, as a result,  we have two different inversion  formulas. The proof relies on the previous formulas   for $F_{m,k}^{(j)} f$, according to which  the restriction $k=m$ implies $j=0$. It means that our reasoning works only for the Funk transform $F_m$ acting from $C^\infty (\vnm)^{O(m)}$  into itself.

 \begin{theorem} \label{jhb67A}  Let  $\vp = F_m f$, $\,f\in C^\infty (\vnm)^{O(m)}$,
$1\le m\le   n-m$. Suppose that $n$ is even and  $ \ell \!=\!(n-2m)/2$.
If
\[D=c \,\Delta_{m-n, \ell}, \qquad  c= \left (\frac{\gm(m/2)}{\Gam_{m}((n\!-\!m)/2)}\right )^2,\]
then
\be\label {encea}
f= D F_m \vp = F_m D \vp.  \ee
\end{theorem}
\begin{proof} In view of Theorem \ref{jhb67},  it remains  to prove the second equality in (\ref{encea}). As above,
$f= \S_{m}^{m-n} f$, where by (\ref{ktyy51}) and Corollary \ref{mcmndfi}  (with $\lam =m-n$),
\[\S_{m}^{m-n} f=\tilde \del   F_m \Cs_{m}^{m-n} f,  \qquad \tilde \del =\frac{\Gam_{m} (m/2)}{\Gam_{m} ((n\!-\!m)/2)}.\]
Using (\ref{lase2hh}) and (\ref{kfskfn}), we can write
\[
\Cs_{m}^{m-n} f=\Delta_{m-n, \ell} \Cs_{m}^{m-n+2\ell} f=\Delta_{m-n, \ell} \Cs_{m}^{-m} f=   \tilde c_0\,\Delta_{m-n, \ell}  F_{m} f,\]
where $\tilde  c_0 = \tilde \del$.
Hence $f=  \tilde \del^2\,F_m \Delta_{m-n, \ell}  F_{m} f=F_m D \vp$.
\end{proof}

Theorem \ref{jhb67A} agrees with the known case $m=1$ for the sphere; see \cite [Theorem 2.6 (i)]{Ru20}.

\section{Conclusion and Open Problems} \label{kjb8j}

In the present paper we  introduced a new family of differential operators on Stiefel (or Grassmann) manifolds and applied these operators to the study of $\lam$-cosine transforms,  Funk transforms, and their intermediate modifications. Our main objective was inversion formulas  on smooth functions. The main tool was  the classical Fourier analysis on matrix space. Of course, many problems are still open. Below we list some of them with the hope that the  reader will be inspired to make further progress.

\vskip 0.2truecm

{\bf 1.} {\it A nonlocal inversion formula for the Funk transform $F_m=F_{m,m}$ in terms of the differential operator
$\Delta_{\lam, \ell}$  and a suitable  back-projection operator, when    $n$ is odd, $1\le m<n/2$}.

 Theorems \ref {jhb67},  \ref{jhbsa67}, do not cover this case.   Nonlocal inversion formulas for the Funk transform on Grassmannians are known in different terms \cite{GGR, GR, Zh1}.

\vskip 0.2truecm

{\bf 2.} {\it  A nonlocal inversion formula for the intermediate  Funk-cosine transform $F_{m,k}^{(j)}$ when $j>0$ and
$n-m+j-k$ is odd};  cf. Theorem \ref {jhb67}, where $n-m+j-k$ is even.

\vskip 0.2truecm

{\bf 3.} {\it An analogue of the equality $\Delta_{\lam, \ell} \stackrel{*}{\Cs}\!{}_{m, k}^{\lam +2\ell} \vp=  \stackrel{*}{\Cs}\!{}_{m, k}^{\lam} \vp$  for $\Cs_{m, k}^{\lam+2\ell}f$ and $\Cs_{m, k}^{\lam}f$, $k>m$};  cf. Theorem \ref{liut0s}.

\vskip 0.2truecm

{\bf 4.} {\it Intertwining formulas for $k>m$ and $j>0$, generalizing  (\ref{encea})}. Formulas of this kind are well known in the Radon transform theory; cf. Theorems 3.1 and 3.8 in \cite[Chapter 1, Section 3]{H11} for the hyperplane Radon transform on $\rn$.

\vskip 0.2truecm

{\bf 5.} {\it  Non-injectivity of $F_{m, k}$ (on right $O(m)$-invariant functions) when $m>k$ or, equivalently,  of $R_{p,q}$ when $p+q>n$. Is it possible to give a relatively simple counterexample?}

 Note that if $k+m\le n$, then, by (\ref{r5yyy}), $\dim \vnm > \dim \vnk$ if and only if $m>k$. Similarly, if $p<q$, then  $p+q>n$ if and only if  $\rank \, G_{n,p} > \rank \, G_{n,q}$.  An outline of the proof,  that $R_{p,q}$ is non-injective when $p+q>n$, was communicated by T. Kakehi \cite{K2} in group representation terms.

\vskip 0.2truecm

{\bf 6.} {\it An analogue of Problem {\bf 5} for intermediate Funk-cosine transforms and their Grassmannian modifications}; {\it see Conjecture \ref {m0opqw}}.

This conjecture resembles  the known fact that the non-normalized $\lam$-cosine transform is non-injective, when $\lam$ is a positive integer; see  \cite{GH, OR05, OR06}.

\vskip 0.2truecm

\appendix

\setcounter{section}{0}

\renewcommand{\thesection}{A}


\section*{Appendix. Smooth Functions on Stiefel and Grassmann  Manifolds}

\vskip 0.2truecm


Below we present necessary information about Stiefel and Grassmann manifolds in a consistent way and prove some auxiliary statements. Our main objective is characterization of smooth functions on these manifolds. Most of this material is scattered in numerous books and papers, according to needs and taste of the authors; see, e.g., \cite{AMS,  Boo, Lee, W83}.  Lemma \ref{wnvp}  is new. We  refer to \cite {Lee} for  terminology.


\subsection{Stiefel  manifolds}

Recall that $\frM_{n,m}$, $n\ge m$, is  the real vector
space of  matrices $x=(x_{i,j})$ having $n$ rows and $m$ columns. We equip $\frM_{n,m}$ with a natural linear manifold structure. A chart of this manifold is given by a map
\[\frM_{n,m}\to\bbr^{nm}, \quad  x \to (x_{1,1}, \ldots, x_{n,1}, x_{1,2}, \ldots, x_{n,2}, \ldots, x_{1,m}, \ldots, x_{n,m}),\]
that  stacks the columns of $x$  below one another. Let  $\tilde\frM_{n,m}$ be the set  of all matrices $x \in \frM_{n,m}$ of rank $m$,  which is an open subset of $\frM_{n,m}$. We consider $\tilde\frM_{n,m}$ as a smooth manifold with the differentiable structure  inherited from  $\frM_{n,m}$.

Let $\vnm= \{v \in \tilde \frM_{n,m}: v'v=I_m \}$
 be  the set of all orthonormal $m$-frames in $\bbr^n$.  Because $v'v=I_m$ gives $m(m+1)/2$ functionally independent polynomial conditions on the $nm$ entries $v_{i,j}$ of $v$, $\vnm$ is an algebraic variety of dimension
 \be\label {r5yyy} d_m=nm-m(m+1)/2.\ee
 It is also a closed subset   of the sphere of radius $\sqrt {m}$ in $\bbr^{nm}$.

 The subset   $\vnm  \subset \tilde\frM_{n,m} $
  can be regarded  as an  embedded submanifold of $\tilde\frM_{n,m}$.
  To prove the latter, consider the  polynomial map \[F: \tilde\frM_{n,m}\to \frM_{m,m}, \qquad F(x)= x^\prime x.\]
 The differential of $F$ is given by
\[DF(x)y=x^\prime y + y^\prime x, \qquad y\in \frM_{n,m}.\]
 It follows that  $F$ has full rank; see, e.g., \cite [p. 26] {AMS} for details. Then,
by the inverse function theorem  \cite[Theorem 1.38]{W83}, $V_{n,m}=F^{-1}(I_m)$
is an embedded submanifold of $\tilde\frM_{n,m}$ with unique  differentiable structure inherited from $\tilde\frM_{n,m}$.  With this manifold structure, the set $\vnm$ is known as the
{\it Stiefel manifold}. Important particular cases are $m=1$ and $m=n$, when $V_{n,1}=\sn$ and $V_{n,n}=O(n)$.

\begin{proposition}\label{mmxzc} The maps
\[O(n) \times \vnm  \to \vnm, \qquad     (g, v) \to g v,\]
\[\vnm \times O(m)  \to \vnm, \qquad     (v,\gam) \to v\gam,\]
 are smooth.
\end{proposition}
\begin{proof} The statement follows, e.g., from  \cite [Corollary 8.25]{Lee}, since, by the above definition, $\vnm$ is an embedded submanifold of $\tilde\frM_{n,m}$ and the maps
\[
F_1, F_2 :\, \vnm \to  \tilde\frM_{n,m},\qquad  F_1: v \to g v, \quad  F_2: v \to v\gam,\]
are smooth. The smoothness of $F_1$ and $F_2$ is obvious  because matrix entries  of $gv$ and $v\gam$  depend polynomially on the matrix entries of  $v$.
\end{proof}


The Stiefel manifold $\vnm$ has several  diffeomorphic realizations. The next one is especially important.  Let us   fix a unit frame
$ v_0= \left[\begin{array} {c}  I_{m} \\  0 \end{array} \right]\in V_{n, m}$.
  The  isotropy subgroup of $O(n)$ at $ v_0$, that can be identified with  $O(n-m)$, is a closed embedded Lie subgroup of $O(n)$.
 Hence the left coset space  $O(n)/O(n-m)$  has a unique smooth manifold structure  such that the quotient map
 \be\label {ropy} \pi: O(n) \to  O(n)/O(n-m)\ee
 is a smooth submersion and  the map
  \be\label {ropy1} F: \, O(n)/O(n-m) \to \vnm, \qquad  F (g  O(n-m)) = gv_0,\ee
 is an $O(n)$-equivariant diffeomorphism; see  \cite[Lemma 9.23 and Theorems 9.22, 9.24]{Lee}. Thus we have the following
 \begin{proposition} \label{jheefe}
 The Stiefel manifold $\vnm$ is diffeomorphic to the quotient manifold $O(n)/O(n-m)$.
 \end{proposition}

 We will need the following general statement.

 \begin{proposition} \label{jhfe} \cite [Proposition 7.17]{Lee}  Suppose $M$, $N$, and $P$ are smooth manifolds,
$\sig: M\to N$ is a surjective submersion, and $f: N\to P$ is any map. Then $f$ is smooth if and only if $f\circ \sig$ is smooth:
\end{proposition}
\[
\xymatrix{&M \ar[d]_\sig \ar[rd]^{f\circ \,\sig} \\
&N \ar[r]_f          &P}
\]

\begin{proposition}\label{thmVNM} Let $ v_0= \left[\begin{array} {c}  I_{m} \\  0 \end{array} \right]\in V_{n, m}$.
 A function $f$ on $\vnm $ is smooth  if and only if a function  $f_0$ on $ O(n)$, defined by $f_0 (g)=
f(gv_0)$, is smooth.
\end{proposition}
\begin{proof}  Let $\kappa: O(n)\to \vnm$, $\kappa (g)=gv_0$. We make use of Proposition \ref{jhfe}  with $M=O(n)$,
$N=\vnm$, $P=\bbc$, and $\sig=\kappa$. The result will be proved if we show that $\kappa$ is a submersion. We have $\kappa=F\circ \pi$, where $\pi$ is the
quotient map (\ref {ropy}) and $F$ is the diffeomorphism (\ref {ropy1}). Because both $\pi$  and $F$ are submersions, their composition is a
 submersion (see, e.g.,  \cite [Exercise 7.2]{Lee}), and the proof is complete.
\end{proof}

We  fix the Haar measure $dv$ on $\vnm$, which is left $O(n)$-invariant,
right $O(m)$-invariant, and
  normalized by \be\label{2.16} \sigma_{n,m}
 \equiv \intl_{\vnm} dv = \frac {2^m \pi^{nm/2}} {\gm
 (n/2)} \ee
\cite[p. 70]{Mu}. The notation $d_\ast v=\sig^{-1}_{n,m} dv$ is used for the corresponding probability measure. For any $v\in \vnm$,
\be\label{porrl} \intl_{\vnm} f(v) \, d_\ast v =\intl_{O(n)}  f(\gam v) \, d_\ast \gam,\ee
where $d_\ast \gam$ stands for the Haar probability measure on $O(n)$.

\begin{proposition} \label {lkutt} If $f$ belongs to $L^1 (\vnm)$ or $C^\infty (\vnm)$, then the average $ f_{ave} (v)=\int_{O(m)} f (v\a)\, d_*\a$ belongs to  $L^1 (\vnm)$ or $C^\infty (\vnm)$, respectively.
  \end{proposition}
\begin{proof} The $L^1$ statement holds by Fubini's theorem. To prove the $C^\infty$ statement, we observe that by  Proposition \ref{mmxzc}, the map $\rho:\, (v, \a) \to v\a$ is smooth. Hence $f (v\a)= (f \circ \rho)(v, \a)$ is  smooth  on $\vnm \times O(m)$, and therefore $f_{ave}$ is a smooth   right $O(m)$-invariant function on $\vnm$.
 \end{proof}

Below we give another  characterization of smooth functions on $\vnm$, which is probably new. Recall that in the case $m=1$, it is  customarily  to define  $C^\infty (\sn) $ as the
space of restrictions onto $\sn$ of $C^\infty $ functions on $\rn \setminus \{0\}$, so that
$f\in C^\infty (\sn) $ if and only if the extended function $\tilde f : x \to f (x/|x|)$ belongs to $C^\infty (\rn \setminus \{0\})$.
The following theorem allows us to proceed in a similar way if $m>1$, when the radial component is an element of the set $\Om_m$ of positive definite symmetric $m\times m$ matrices.

\begin{theorem}\label{l2.3}  \cite{FK, Herz,  Mu} {\rm (Matrix Polar Decomposition)} If $n  \ge  m$, then every matrix  $x \in  \tilde\frM_{n,m}$  can be uniquely represented as
\be\label{pol}
x=vr^{1/2}, \qquad v \in \vnm,   \qquad
r=x'x \in\Omega_m,\ee
and $dx=2^{-m} |r|^{(n-m-1)/2} dr dv$.
\end{theorem}

By this theorem, there is a one-to-one correspondence between functions $f$ on $\vnm$ and their homogeneous extensions
\[\tilde f(x) =f (x (x'x)^{-1/2}), \qquad x \in \tilde\frM_{n,m}.\]

\begin{lemma} \label{wnvp} The  relations $f\in C^\infty (\vnm) $ and $\tilde f \in C^\infty (\tilde\frM_{n,m})$ are equivalent.
\end{lemma}
\begin{proof}
 We make use of Proposition \ref{jhfe} with $M=\tilde\frM_{n,m}$, $N=\vnm$, $P=\bbc$, and $\sig (x)= x (x'x)^{-1/2}$. It suffices to show that $\sig :  \tilde\frM_{n,m} \to \vnm $ is a surjective submersion.
 Clearly, $\sig$ is smooth and its differential, as a linear map between  tangent spaces at $x \in \tilde\frM_{n,m}$ and $v=\sig (x) \in \vnm$, has full rank, which is equal to $\dim \vnm$. Thus $\sig$ is a smooth surjective map of constant rank, and therefore (use, e.g. \cite [Theorem 7.15]{Lee}) it is a submersion. This completes the proof.
\end{proof}

\begin{remark} \label {kjutr}  By Lemma \ref{wnvp}, we can realize $C^\infty (\vnm) $ as the space of all functions $f$ on $\vnm$, for which $\tilde f \in C^\infty (\tilde\frM_{n,m})$ in the usual sense, as on $\bbr^{nm}$.
 This remark plays a key role in the paper because it allows us to define differential operators on $\vnm$ via homogeneous continuation.
\end{remark}

 \subsection{Grassmann manifolds}
We denote by $G_{n,m}$ the Grassmann manifold of  $m$-dimensional linear subspaces of $\bbr^n$.
 There exist several diffeomorphic realizations of $G_{n,m}$.
see, e.g., \cite [pp. 22, 234, 238 (Problem 9-14)]{Lee}. The Lie group $G=O(n)$ acts on $G_{n,m}$ smoothly  and transitively, and therefore $G_{n,m}$ is a homogeneous $G$-space.
If $\xi_0 =\span (e_1, \ldots, e_m) \in G_{n,m}$ is the coordinate subspace of $\rn$,
then $G_0=O(n-m) \times O(m)$ is the isotropy group of $\xi_0$,  which is a closed embedded Lie subgroup of $G$ (see, e.g., \cite[Lemma 9.23]{Lee}).  Hence Theorem 9.24 from \cite{Lee} yields the following

\begin{proposition} \label {dveedf} The maps
\be\label {eedf} E_1 : G/G_0  \to G_{n,m}, \qquad  E_2 : G/G_0  \to G_{n,n-m}\ee
 are $G$-equivariant diffeomorphisms.
 \end{proposition}

 By this proposition,   $G_{n,m}$ and $G_{n,n-m}$ are  diffeomorphic and every function $f_1$ on $G_{n,m}$ can be identified with a function $f_2$ on $G_{n,n-m}$,  so that
 \[f_1 (\xi)=f_2 (\xi^\perp), \qquad f_2(\eta)=f_1 (\eta^\perp); \qquad  \xi \in G_{n,m}, \quad  \eta \in G_{n,n-m}.\]
In other words,
\[
f_1 (\xi)=f_2(E_2 (E_1^{-1} \xi), \qquad f_2(\eta)=f_1(E_1 (E_2^{-1} \eta).\]
This reasoning gives the following
\begin{proposition}\label{mofoth} $f_1\in C^\infty (G_{n,m})$ if and only if $f_2\in C^\infty (G_{n,n-m})$.
\end{proposition}

  \begin{proposition} \label{wnvpc} Given a function $g$ on $G_{n,m}$, let
\[ g_0 (\gam) = g (\gam \xi_0), \qquad \gam \in O(n),\qquad  \xi_0 =\span (e_1, \ldots, e_m).\]
 Then $g_0 \in C^\infty (O(n))$ if and only if    $ g\in C^\infty (G_{n,m})$.
\end{proposition}
\begin{proof}  Let $\kappa: O(n)\to G_{n,m}$, $\kappa (g)=g\xi_0$. We make use of Proposition \ref{jhfe}  with $M=O(n)$, $N=G_{n,m}$, $P=\bbc$, and $\sig=\kappa$. The result will be
proved if we show that $\kappa$ is a submersion. Denote $G=O(n)$, $G_0=O(n-m)\times O(m)$.
We have $\kappa=E_1\circ \pi$, where $\pi :  G\to G/G_0$ is the quotient map  and $E_1$ is the diffeomorphism from (\ref{eedf}).
Because both $\pi$  and $E_1$ are submersions, their composition is a submersion, and the proof is complete.
\end{proof}

 Every subspace $\xi \in G_{n,m}$ is uniquely determined by its orthonormal basis $v \in \vnm$. Because all  bases of the form $v\gam$, $\gam \in O(m)$, define the same subspace, we can
 realize $G_{n,m}$  as a quotient space
\[G_{n,m} \simeq V_{n,m}/O(m). \]

 \begin{proposition} \label{wnsvpc} The map
\[E_3: \,  G_{n,m} \to V_{n,m}/O(m)\]
is an $O(n)$-equivariant diffeomorphism.
\end{proposition}
\begin{proof} Let $G=O(n)$. By Proposition \ref{dveedf}, it suffices to prove this statement with $G_{n,m}$ replaced by  $G/G_0$,  $G_0=O(n-m) \times O(m)$.  The manifold $V_{n,m}/O(m)$
is a homogeneous $G$-space with the isotropy subsgroup $G_0$  of the element $v_0 O(m)\in V_{n,m}/O(m)$, $ v_0= \left[\begin{array} {c}  I_{m} \\  0 \end{array} \right]\in V_{n, m}$.
 Hence, by Theorem 9.24 from \cite{Lee},  there exists a $G$-equivariant diffeomorphism  between $G/G_0$  and  $V_{n,m}/O(m)$. This completes the proof.
\end{proof}

\begin{remark} \label {kjutrgr} By Proposition \ref {wnsvpc}, we can identify functions  $g\in C^\infty (G_{n,m}) $ with functions
 $f: v \to g(\{v\})$ belonging to $C^\infty (\vnm)^{O(m)}$. Thus, in view of Remark \ref{kjutr}, there is a one-to-one correspondence between smooth functions on the Grassmannian  $G_{n,m}$ and smooth functions on the matrix space $\tilde\frM_{n,m}$.
 \end{remark}

 Owing to  diffeomorphisms
   \[  V_{n, m}/O(m) \Longleftrightarrow G_{n, m} \Longleftrightarrow  G_{n, n-m} \Longleftrightarrow V_{n, n-m}/O(n-m),\]
 we obtain the following statement.
\begin{proposition}\label {mnxmnb1} There is  a one-to-one  correspondence
\be\label {wweed}  f   \simeq f_*\ee
 between   right $O(m)$-invariant  functions $f$ on $\vnm$ and right $O(n-m)$-invariant
 functions $f_*$ on $V_{n, n-m}$. Moreover,
$f\in C^\infty (\vnm)^{O(m)}$ if and only if $f_*\in C^\infty (V_{n, n-m})^{O(n-m)}$.
\end{proposition}

The next Proposition, which is a consequence of normalization,   characterizes integrability properties of functions
 $f, f^*$, and the relevant functions on Grassmannians.

\begin{proposition}\label {muu1}  Given $v \in V_{n,m}$, let $\tilde v\in V_{n, n-m} $ be an arbitrary frame,  which is orthogonal to the subspace  $ \xi=\{v\}$, and let $\tilde \xi =\{\tilde v\}$.
 If $f$ is a right $O(m)$-invariant  function on $\vnm$, $f^*$ is defined by
(\ref{wweed}),  and the functions $g$ and $g^*$ are defined by
\[f(v)=g(\{v\}), \qquad f^*(\tilde v)=g^*(\{\tilde v\}),\] then
\be\label {wgted}
\intl_{V_{n,m}} \!\!f(v)\, d_*v \!=\! \intl_{V_{n,n-m}}\! \!\!f^*(\tilde v)\, d_*\tilde v \!=\! \intl_{G_{n,m}} \!\!g(\xi)\, d_*\xi\! = \!\intl_{G_{n,n-m}} \!\!\!g^*(\tilde \xi)\, d_*\tilde \xi,\ee
provided that at least one of these integrals exists in the Lebesgue sense.
\end{proposition}

\vskip 0.2 truecm

\noindent {\bf Acknowledgement.} I would like to thank Semyon Alesker,  Fulton Gonzalez, Tomoyuki Kakehi,  Gestur \'Olafsson,  Isaac Pesenson, and Siddhartha Sahi for helpful discussions.


\end{document}